\documentclass[12pt]{amsart}
 \usepackage{amsmath,amssymb,amsfonts,amsthm,amscd,textcomp,times,booktabs}
  \ifx\pdfoutput\undefined
   \usepackage[dvipdfmx]{graphicx}
   \else
   \fi
  \usepackage{braket}
  \usepackage{color,float}
   \usepackage[normalem]{ulem}
\newtheorem{theorem}{Theorem}

\newtheorem{proposition}[theorem]{Proposition}
\newtheorem{lemma}[theorem]{Lemma}
\newtheorem{definition}[theorem]{Definition}
\newtheorem{corollary}[theorem]{Corollary}
\newtheorem{remark}[theorem]{Remark}

\newtheorem{problem}[theorem]{Problem}

\numberwithin{equation}{section} \numberwithin{theorem}{section}
\usepackage[all]{xy}

\newcommand{\mapright}[1]{\smash{\mathop{   \hbox to 0.7cm{\rightarrowfill}}
 \limits^{#1}}}

\newcommand{\Z}{\ensuremath{\mathbb{Z}}}
\newcommand{\Q}{\ensuremath{\mathbb{Q}}}
\newcommand{\R}{\ensuremath{\mathbb{R}}}
\newcommand{\C}{\ensuremath{\mathbb{C}}}

\newcommand{\Vol}{\ensuremath{\mathrm{Vol}}}
\newcommand{\bracket}[1]{\ensuremath{\langle #1 \rangle}}

\DeclareMathOperator{\Hom}{Hom}

\def\C{\mathbb C}

\def\D{\Delta}

\def\p{\partial}

\def\P{\mathbb{P}}
\def\R{{\mathbb R}}

\begin{document}
%%%%%%%%%%%%%%%%%%%%%%%%%%%%%%%%%%%%%%%%%%%%%%%%%%%%
\title{Relative Algebro-Geometric stabilities of Toric Manifolds}

\author{Naoto Yotsutani}

%\address{School of Mathematical Sciences at Fudan University,Shanghai, 200433, P. R. China}
\address{Graduate School of Mathematics, Nagoya University, Nagoya, 464-8602, Japan}
%\email{naoto-yotsutani@fudan.edu.cn}
\email{naoto.yotsutani@gmail.com}

\author{Bin Zhou}
\address{School of Mathematical Sciences, Peking University,
Beijing, 100871, P. R. China; and
Mathematical Sciences Institute, The Australian National University, Canberra, ACT 2601, Australia.}

\email{bzhou@pku.edu.cn}

\thanks {The second author was supported partially by ARC grant DE120101167,  NSFC grants No. 11571018 and 11331001.}

\makeatletter
\@namedef{subjclassname@2020}{%
  \textup{2020} Mathematics Subject Classification}
\makeatother

%\dedicatory{}
%%%%%%%%%%%%%%%%%%%%%%%%%%%%%%%%%%%%%%%%%%%%%%%%%%%%%%
%\date{\today}
\maketitle

%%%%%%%%%%%%%%%%%%%%%%%%%%%%%%%%%%%%%%%%%%%%%%%%%%

\begin{abstract}
In this paper we study the relative Chow and $K$-stability of toric manifolds in the toric sense. 
First, we give a criterion for relative $K$-stability and instability of toric Fano manifolds in the toric sense. 
The reduction of relative Chow stability on toric manifolds will be investigated using the Hibert-Mumford criterion in two ways.
One is to consider the maximal torus action and its weight polytope. 
We obtain a reduction by the strategy of Ono \cite{Ono13}, which fits into the relative GIT stability detected by Sz\'ekelyhidi.
The other way relies on $\C^*$-actions and Chow weights associated to 
toric degenerations following Donaldson and Ross-Thomas \cite{D02, RT07}.  
As applications of our main theorem, we partially determine the relative $K$-stability of  toric Fano threefolds and present 
counter-examples which are relatively $K$-stable in the toric sense but which are asymptotically relatively Chow unstable.
In the end, we explain the erroneous parts of the published version of this article (corresponding to Sections $\ref{sec:Intro}$-$\ref{sec:example}$), which provides some inconclusive results for relative $K$-stability in Table $\ref{list-stability}$.
\end{abstract}

\section{Introduction}\label{sec:Intro}

%sec1

\vskip 10pt

The well-known Yau-Tian-Donaldson conjecture 
asserts that a compact complex polarized manifold $(X, L)$ admits canonical metrics (K\"ahler-Einstein metrics, constant scalar curvature (cscK) metrics, and extremal metrics, etc)
in $2\pi c_1(L)$ if and only if 
$(X,L)$ is stable in the sense of Geometric Invariant Theory.
Among various notions of stability, 
$K$-stability and Chow stability are the most widely studied. 
Many authors use the term {\it polystability} rather than {\it stability}, 
since the former agrees better with the notions in GIT. Throughout this paper, we use the latter for simplicity.

The conception of $K$-stability was first introduced by Tian  \cite{Ti97} in the study of the existence of 
K\"ahler-Einstein metrics in the first Chern class (if it is positive) on a K\"ahler manifold.  
Later, Donaldson extended it to general polarized varieties \cite{D02} 
and made a conjecture on the relation between $K$-stability and the existence of 
cscK metrics. More generally, for the existence of extremal metrics, 
the definition of $K$-stability was extended by Sz\'ekelyhidi \cite{Sz07} to K\"ahler classes
with 
non-vanishing Futaki invariant and was called relative $K$-stability. 
Meanwhile, the conception of Chow stability is also significant in K\"ahler geometry. Let
${\mathrm{Aut}}(X,L)$ be the automorphism group of $(X,L)$.
In \cite{D01} Donaldson showed that the existence of a cscK metric 
in $2\pi c_1(L)$ implies the asymptotic Chow stability of $(X,L)$ if ${\mathrm{Aut}}(X,L)$ is discrete. 
Donaldson's result was generalized by Mabuchi \cite{Mab05}, with the assumption on ${\mathrm{Aut}}(X,L)$ replaced by  the condition of vanishing higher order Futaki invariants. Very recently, it has been shown that the existence
of extremal metrics implies asymptotically Chow stability relative to a maximal torus \cite{Mab16, S16}.
With these
remarkable progress, the verification of the stabilities is drawing more and more attention. 
In general, this is a complicated problem since one has to study an infinite number of possible degenerations of the manifold. In this paper, we shall discuss the stabilities of toric manifolds.

For toric manifolds, 
a well-understood reduced version of the relative $K$-stability 
on the moment polytope is believed to
be equivalent to the existence of extremal metrics \cite{D02, ZZ08}.  
This conjecture has been confirmed for
toric surfaces \cite{D08, D09, CLS}. 

Let $(X_\D, L_\D)$ be the polarized toric manifold which corresponds to a lattice polytope $\D$
 \begin{equation}
 \langle l_i, x\rangle \leq \lambda_i,  i=1,\ldots,d,
\end{equation}
 satisfying Delzant's condition, where $\lambda_i \in \Z$, $l_i\in \mathbb Z^n$ is primitive.
%Without loss of generality, we assume $0$ lies in the interior of $\D$.
Let $\theta_\D$ be the   
potential function of the extremal vector field $V$\cite{FM95}, which is affine linear on $\D$, and
 normalized by $\int_\D \theta_\D\,dx=0$ (see Lemma \ref{lemma:extremal}).
In \cite{D02}, Donaldson reduced $K$-stability  to 
the positivity of a linear functional defined on $\D$. 
This functional was generalized for  
relative $K$-stability and is given by \cite{ZZ08}
\begin{equation}\label{linearfunc}
\mathcal L_\D(u)=\int_{\p\Delta}u\,d\sigma-\int_\Delta(\bar S+\theta_\D) u\,dx
\end{equation}
where $\bar S$ is the average of the scalar curvature, and $d\sigma=|l_i|^{-1}d\sigma_0$ on the face in $\{x\in\mathbb R^n: \langle l_i, x\rangle = \lambda_i\}$. 
Here $d\sigma_0$ is the standard Lebesgue measure on $\p\D$.
Note that $ \bar S=\frac{\Vol (\p \D)}{\Vol (\D)}$ and the
functional $\mathcal L_\D$ corresponds to the modified Futaki invariant in \cite{Sz07}.
We mention that the potential function $\theta_\D$ is uniquely determined by the condition that
$\mathcal L_{\D}(u)=0$ for any affine linear function $u$, namely, one can solve the $n+1$-linear system
\[
\mathcal L_{\D}(1)=0, \quad \mathcal L_{\D}(x_i)=0 \qquad \text{for} \qquad i=1, \dots, n,
\]
in order to find $\theta_\D=\sum a_ix_i + c$ with $a_i$ and $c$.
See Section $\ref{sec:computation}$ for more detail.

Recall that a convex function $u$ is {\emph{piecewise linear}} 
if there are affine linear functions $f_1,\dots , f_{\ell}$ such
that $u=\mathrm{max}\set{f_1,\dots ,f_{\ell}}$.
Furthermore $u$ is {\emph{simple piecewise linear}}
if it is of the form $u=\mathrm{max}\set{0,f}$ for a linear function $f$. In view of \cite{D02, Sz07, ZZ08, ZZ08-2},
we have:

\begin{definition}\rm\label{definition:relK}
A toric manifold $(X_\Delta, L_\D)$ is called {\it relatively $K$-semistable  in the toric sense}
if $\mathcal L_\D(u)\geqslant 0$ for all piecewise linear convex functions. 
Furthermore, it is called {\it relatively $K$-stable in the toric sense} 
when $\mathcal L_\D(u)=0$ if and only if $u$ is affine linear.
\end{definition}

When proving the existence of  cscK metrics on toric surfaces 
\cite{D08, D09} Donaldson introduced a stronger notion called {\it strong $K$-stability}. 
Let $\D^*$ be the union of  the interior of $\D$ and the interiors of its co-dimension $1$ faces. 
Denote 
\[
\mathcal C_1=\{u\ |\ \text{$u$ is convex on $\D^*$ and $\int_{\partial \D} u< \infty$} \}.
\] 
The linear functional $\mathcal L_\D$ is well-defined on 
$\mathcal C_1$.

\begin{definition}\rm\label{definition:unif-relK}$(X_\D, L_\D)$ is called
{\it relatively strongly $K$-stable in the toric sense} if $\mathcal L_\D(u)\geqslant 0$ 
for all convex functions
in $\mathcal C_1$ and $\mathcal L_\D(u)=0$ if and only if $u$ is an affine linear function.
\end{definition}

To check the positivity of  $\mathcal L_\D$ in dimension $2$, 
it suffices to consider simple piecewise linear convex functions 
on $\Delta$ \cite{D02, WZ11}. 
In higher dimensions,
a sufficient condition for the relative strong $K$-stability in the toric sense was given by \cite{ZZ08}.
When $X$ is a Fano $n$-fold,  
$L$ is the anti-canonical line bundle and $\bar S=n$ the sufficient
condition is  
\begin{equation}\label{ineq:normal_pot}
\sup_\Delta \theta_\D  \leqslant 1.
\end{equation}
Condition (\ref{ineq:normal_pot}) has been verified for all toric Fano surfaces. In \cite{ZZ08}, it 
was also asked whether it holds 
for higher dimensions or not. Furthermore, 
Mabuchi proposed the following question when considering the existence
of extremal metrics on toric Fano manifolds.
\begin{problem}
Let $(X_\D,L_\D)$ be a polarized toric Fano manifold.
Is $(X_\D,L_\D)$ always relatively $K$-stable or not?
\end{problem}
If the answer were affirmative, one can expect that any toric Fano manifold 
admits an extremal metric like the case of
 K\"ahler-Ricci solitons \cite{WZ04}. However, we 
have found counter-examples (Corollary $\ref{cor:unstableFano3}$).
By a simple observation and together with (\ref{ineq:normal_pot}), we have:

\begin{theorem}\label{prop:unstablecond}
Assume  $(X_\D, L_\D)$ is a toric Fano manifold and $\theta_\D=\sum_{i=1}^na_ix_i+c$, where $a_i, c\in \R$. Let $\Delta^{-}=\set{x\in\Delta: 1-\theta_\D<0 }$. 

\begin{enumerate}
\item If ${\rm Int}(\Delta^{-})=\emptyset$, i.e. $\theta_{\D}\leq 1$ on $\D$, then $(X_\D, L_\D)$  is 
 relatively strongly $K$-stable in the toric sense. Here ${\rm Int}(\Delta^{-})$ is the interior of $\D^-$. 

\item If ${\rm Int}(\Delta^{-})\neq\emptyset$ and satisfies
\begin{equation}\label{unstablecond}
1-c<\frac{\int_{\Delta^{-}} (1-\theta_\D)^2\,dx}{{\mathrm{Vol}}(\Delta^{-})},
\end{equation}
then there exists a simple piecewise linear function such that $\mathcal L_\D(u)<0$.
\end{enumerate}
\end{theorem}

An application of   
Theorem $\ref{prop:unstablecond}$ is to determine relative $K$-stablities of 
all toric Fano threefolds. 
The condition 
$\theta_\D\equiv 0$ is equivalent to the vanishing of the Futaki invariant  \cite{FM95}.  
Smooth toric Fano threefolds were classified by Batyrev \cite{B81, B98} and K. Watanabe and
M. Watanabe \cite{WW82} independently. In this paper, we use the notation of \cite{B98}.
See Section $\ref{sec:Toric}$ for more detail.
Among all of them, $\C P^3$, $\mathcal{B}_4$, $\mathcal{C}_3$, 
$\mathcal{C}_5$ and $\mathcal{F}_1$ have 
a vanishing Futaki invariant, so
condition  $\eqref{ineq:normal_pot}$ is 
 true. By computation with Theorem \ref{prop:unstablecond}, we have:

\begin{theorem}\label{thm:main}
Let $X$ be a toric Fano threefold. We assume that the Futaki invariant of $X$ does not vanish.
Then $X$ is relatively strongly $K$-stable in the toric sense in the anti-canonical class if and only if $X$ is one of the following:
$\mathcal{B}_2$, $\mathcal{B}_3$, $\mathcal{C}_1$, $\mathcal{C}_4$, $\mathcal{E}_3$,
$\mathcal{E}_4$ and $\mathcal{F}_2$.
\end{theorem}

It is known that all toric Fano surfaces admit extremal metrics in the anti-canonical class \cite{Ca82, CLW08}. The instability tells us that counter-examples appear in dimension 3.

\begin{corollary}\label{cor:unstableFano3}
If $X$ is one of $\mathcal{B}_1$, $\mathcal{C}_2$, $\mathcal{D}_1$, $\mathcal{D}_2$, $\mathcal{E}_1$ and $\mathcal{E}_2$, then $X$ does not admit extremal metrics in its first Chern class.
\end{corollary}

On the other hand,  the reduction of Chow stability is also an interesting problem.  
A natural idea is to use the 
Hilbert-Mumford criterion.
Ono  \cite{Ono13} studied Chow stability of toric manifolds 
by considering the maximal torus action and its weight polytope.
He obtained a reduction by adapting 
Gelfand-Kapranov-Zelevinsky's theory of Chow polytopes \cite{GKZ94, KSZ92}. 
He also defined a notion of the relative Chow semistability in the toric sense. 
In this paper, we introduce a refinement of this notion so that it fits naturally into 
the relative GIT stability detected by Sz\'ekelyhidi \cite[Chapter $1$]{SzThesis}.

Let $(X_\D, L_\D)$ be a polarized toric manifold and $N=\dim (H^0(X_\D, L_\D))-1$. 
We consider the relative Chow stability of $X_\Delta\subset \C P^N$(see Section \ref{sec:lowest} for definitions).
Now we assume $G=(\C^{*})^{N+1}$ is a subgroup of diagonal matrices in $\mathrm{GL}(N+1, \C)$.
Following \cite{Ono13}, we only consider the specific maximal torus of $\mathrm{SL}(N+1, \C)$
which is also a subtorus of $(\C^{*})^{N+1}$ given by
\[
T_\D^{\C}=\set{(t_1,\dots, t_{N+1})\in (\C^{*})^{N+1} | \prod_{j=1}^{N+1}t_j=1}.
\]
Let $\beta$ be the $\C^*$-action induced by the extremal vector field $V$ as in Section $\ref{redK}$.
\begin{definition}\rm\label{toricrelchow}
$(X_\D, L_\D)$ is {\it relatively Chow semistable(stable, unstable) in the toric sense} if the Chow form is $T_\D^{\C}$-semistable(stable, unstable) relative to $\beta$.
\end{definition}

Finally we consider the asymptotic relative Chow stability.
Denote the Ehrhart polynomial of $\D$ by $E_{\D}(t)$.
For any $i\in \Z_+$, 
we replace $\D$ above by $i\D$, $N+1$ 
by $E_\D(i)$  and $G=(\C^*)^{E_\D(i)}$. Then  
we consider the maximal diagonalized torus 
$T_{i\D}^{\C}:=G\cap {\mathrm{SL}}(E_\D(i),\C)$.
\begin{definition}\rm\label{toricrelchow-asy}
$(X_\D, L_\D)$ is {\it asymptotically relatively Chow semistable (stable, unstable) in the toric sense} if the Chow form is $T_{i\D}^{\C}$-semistable (stable, unstable) relative to $\beta$ for all sufficiently large $i$.
\end{definition}

In this paper, we will describe the asymptotic relative Chow stability in the toric sense in a combinatorial way. The character group $\chi(G)$ of %$(\C^*)^{E_\D(i)}$ 
$G$ is identified with 
$$\{i\D\cap \Z^n\rightarrow \Z\}\cong \{\D\cap (\Z/i)^n\rightarrow \Z\}\cong \Z^{E_\D(i)}.$$
For future convenience, we denote $\chi(G)\otimes \R$ by
\begin{eqnarray}\label{eq:SetOfCharFun}
W(i\D):=\{i\D\cap \Z^n\rightarrow \R\}\cong \{\D\cap (\Z/i)^n\rightarrow \R\}\cong \R^{E_\D(i)}.
\end{eqnarray}
As in \cite[p.$220$]{GKZ94}, we identify $W(i\D)$ with its dual space by the scalar product
\[
\langle\varphi, \psi\rangle=\displaystyle \sum_{\mathbf{a}\in \D\cap (\Z/i)^n}\varphi(\mathbf{a})\psi(\mathbf{a}).
\]
Let 
$ \bar \theta_{i\D} =\frac{1}{E_\D(i)} \sum_{\mathbf{a}\in \D\cap (\Z/i)^n}\theta_\D(\frac{\mathbf{a}}{i})$.
We define $d_{i\D}, \tilde{\theta}_{i\D}\in W(i\D)$ by 
\[d_{i\D}(\mathbf{a})=1, \ \tilde{\theta}_{i\D}(\mathbf{a})=\frac{\theta_\D(\mathbf{a})-\bar{\theta}_{i\D}}{i},\ \ 
\  \mathbf{a}\in \D\cap (\Z/i)^n.\] 

\begin{theorem}\label{def:ARCSStoric}\rm
$(X_\D, L_\D)$ is asymptotically relatively Chow semistable in the toric sense if there is an $i_0$ such that for each $i\geqslant i_0$,  there exists $s_i$ satisfying
\begin{equation}\label{assy-rel-chow}
\frac{i^n(n+1)!\mathrm{Vol}(\D)}{E_{\D}(i)}\left(d_{i\D}+s_i\tilde{\theta}_{i\D}\right)\in {\mathrm{Ch}}(i\D),
\end{equation}
where $\mathrm{Ch}(i\D)\subset W(i\D)$ is the Chow polytope of $(X_\D, L_{\D}^i)$.
Furthermore, it is asymptotically relatively Chow stable in the toric sense if
\begin{equation}\label{assy-rel-chow-1}
\frac{i^n(n+1)!\mathrm{Vol}(\D)}{E_{\D}(i)}\left(d_{i\D}+s_i\tilde{\theta}_{i\D}\right)\in {\rm Int(\rm{Ch}}(i\D)).
\end{equation}
\end{theorem}

\begin{remark}\rm
As can be shown (Remark \ref{rem:necAsympChow}(1)), $s_i$, if exists, can be explicitly given by
\begin{equation}\label{s-explicit}
s_i=\frac{i\bar\theta_{i\D}E_\D(i)}{\displaystyle\sum_{\mathbf{a}\in \D \cap (\Z / i)^n}  (\theta_\D(\mathbf{a})-\bar\theta_{i\D})^2}.
\end{equation}
\end{remark}

\vskip 10pt

In \cite{RT07}, Ross-Thomas gave a description of Chow stability by using the Hilbert-Mumford criterion 
for the $\C^*$-actions induced by test configurations \cite{D02}. 
Inspired by this idea, we give an alternative
reduction of the relative Chow stability of toric manifolds in Section \ref{redu-2}. 
In order to see its relation to relative $K$-stability,
we define 
\begin{equation}\label{eq:funcQ}
\mathcal Q_\D(i, g)=E_\D(i)\int_\D g\, dx -\mathrm{Vol}(\D) \sum_{{\mathbf{a}}\in \D \cap (\Z / i)^n} \left(1+s_i\tilde\theta_{i\D}(\mathbf{a})\right)g(\mathbf{a})
\end{equation}
for any $g\in PL(\D,i)$. Here  $PL(\D,i)$ is the subset of piecewise linear concave functions (see Section \ref{chow-reduction}) and $s_i$ is given by \eqref{s-explicit}. Then we have

\begin{theorem}\label{rel-pl}
For any $i\in \Z_{+}$, 
$(X_\D, L_\D^i)$ is relatively Chow semistable in  the toric sense if and only if
$\mathcal Q_\D(i,g)\geqslant 0$,
for all $g\in PL(\D, i)$. 
In addition, it is relatively Chow stable in the toric sense if the equality holds
only if $g$ is an affine linear function.
\end{theorem}

We would like to point out that Theorems 
\ref{def:ARCSStoric}, \ref{rel-pl} also hold for general polarized toric  varieties.
Concerning on relation between Chow and $K$-stabilities, we have:

\begin{table}\label{list-stability}
\caption{Relative stability in the toric sense of toric Fano threefolds}\label{table:toricFanoRelSta}
\begin{center}
\begin{tabular}{ccc}
\toprule
Notation  & Relative $K$-stability & Asymptotic relative Chow stability\\ 
& & (Definition $\ref{def:ARCSStoric}$) \\ \midrule
$\C P^3$  & \underline{stable}  &  \underline{stable} \\ [3pt]
$\mathcal{B}_1$ & unstable  & unstable \\ [3pt] 
$\mathcal{B}_2$  &   stable &  \\ [3pt]
$\mathcal{B}_3$  &  stable & \\ [3pt]
$\mathcal{B}_4$ & \underline{stable} &  \underline{stable} \\ [3pt]
$\mathcal{C}_1$ &  stable &  \\ [3pt]
$\mathcal{C}_2$ & unstable & unstable \\ [3pt]
$\mathcal{C}_3$ & \underline{stable} &  \underline{stable} \\ [3pt]
$\mathcal{C}_4$  & stable&  \\ [3pt]
$\mathcal{C}_5$ & \underline{stable} &  \underline{stable} \\ [3pt]
$\mathcal{D}_1$ & unstable & unstable \\ [3pt]
$\mathcal{D}_2$ & unstable & unstable \\ [3pt]
$\mathcal{E}_1$ & unstable & unstable \\ [3pt]
$\mathcal{E}_2$ & unstable & unstable \\ [3pt]
$\mathcal{E}_3$ &  stable &  \\ [3pt]
$\mathcal{E}_4$ &  stable & unstable \\ [3pt] 
$\mathcal{F}_1$ & \underline{stable} & \underline{stable} \\ [3pt] 
$\mathcal{F}_2$ &  stable &  \\
\bottomrule
\end{tabular}

\vskip 5pt

\begin{enumerate}
\item[(1)] All \underline{stable} in the table are known before. Others are new in the present paper. All \underline{stable} in relative $K$-stability follows from \cite{ZZ08-2}. All \underline{stable} in asymptotic relative Chow stability (Definition $\ref{def:ARCSStoric}$) follows from \cite{Mab05} and \cite{FOS11}.
\item[(2)] Other relative $K$-stability/instability follows from Thereom $\ref{prop:unstablecond}$.( Proposition $\ref{prop:toricFano}$).
\item[(3)] Asymptotic relative Chow unstability except for $\mathcal{E}_4$ follows from Theorem $\ref{thm:relChowKss}$.
\item[(4)] Asymptotic relative Chow unstability of $\mathcal{E}_4$ follows from Proposition $\ref{prop:relChowunst3}$.
\end{enumerate}
\end{center}
\end{table}

\begin{theorem}\label{thm:relChowKss}
If a polarized toric manifold $(X, L)$ is asymptotically relatively Chow semistable in the toric sense, then it is relatively $K$-semistable in the toric sense. 
\end{theorem}

In view of Theorem \ref{thm:main}, we also have the following.
\begin{corollary}
If $X$ is one of $\mathcal{B}_1$, $\mathcal{C}_2$, $\mathcal{D}_1$, $\mathcal{D}_2$, $\mathcal{E}_1$ and $\mathcal{E}_2$, then $X$ 
is asymptotically relatively Chow unstable.
\end{corollary}

In general, asymptotic Chow semistability is much stronger than $K$-semistability.
In order to see the direct evidence of the difference between Chow stability and $K$-stability 
consider the first counter-example that was discovered in \cite{OSY12}.
They used the non-symmetric K\"ahler-Einstein toric Fano
$7$-fold of
\cite{NP11}. In the case where $X$ is non-toric,
lower dimensional counter-examples were discovered by 
Odaka \cite{Oda12} and Vedova and Zuddas \cite{VZ12}.
In \cite{NP11}, it was also proved that all toric Fano manifolds with the vanishing Futaki invariant 
are symmetric if $\dim X \leqslant 6$.
Note that if $X$ is a symmetric toric Fano manifold, 
then $(X,-K_X)$ is asymptotically Chow stable \cite{Mab05, Fut04}.
Hence the lowest dimension for an anti-canonically polarized K\"ahler-Einstein 
toric Fano manifold $(X,-K_X)$ to be asymptotically Chow unstable is $7$. 
One aim of this paper is to provide such an example in a lower dimensional toric case. 
We have found a 3-dimensional toric orbifold example admitting the K\"ahler-Einstein metric
but which is asymptotically Chow unstable in the case where $X$ is $\Q$-Fano (Proposition $\ref{thm:Chowunst}$).
When we consider the relative stabilities, we find the smooth example $\mathcal E_4$ which is relatively $K$-stable but not asymptotically relatively Chow semistable (Proposition $\ref{prop:relChowunst3}$).
The asymptotic Chow stability of $\C P^3$, $\mathcal{B}_4$, $\mathcal{C}_3$, $\mathcal{C}_5$, $\mathcal{F}_1$ follows from \cite{FOS11, Mab05}. Hence, we list all the determined stability of
toric Fano threefolds in this paper in Table \ref{list-stability}. Note that the stabilities are all in the toric sense.
It is an interesting question to complete the table, i.e. to determine the  
remaining stabilities.

\vskip 8pt

This paper is organized as follows. Section $\ref{sec:prelim}$ 
is a brief review of toric varieties and the reduction of 
relative $K$-stability on toric manifolds which will be used at later stages in the paper.
We also prove Theorem $\ref{prop:unstablecond}$.
In Sections $\ref{sec:Chowsta}$ and  $\ref{redu-2}$ we shall discuss the two ways of reduction of 
the relative Chow stability on toric manifolds.
In Section \ref{sec:example}, we present various examples for the stabilities considered in the paper.
We compute normalized potentials on toric Fano threefolds  and verify
the relative $K$-stability or instablity in Section $\ref{sec:computation}$. 
We also provide an example of $K$-stable toric Fano orbifold $X$ 
which is asymptotically Chow unstable in $\dim X=3$. 
Finally, we discuss the asymptotic relative Chow stability of toric Fano threefolds. 
The computational results for $\theta_\D$ and $\D^-$ are listed in Table $\ref{table:toricFano3}$.

\vskip 5pt

\noindent {\bfseries Acknowledgements.} 
The first author would like to thank Professors Y. Nakagawa, Y. Sano and A. Higashitani 
for their valuable comments and discussions. 
In particular, Higashitani suggested to us to use the toric package \cite{normaliz} for our computations. 
Both authors thank the referee for valuable suggestions to improve this article.

%sec2

\section{Preliminaries}\label{sec:prelim}

%subsec 2.1

\subsection{Toric varieties}\label{sec:Toric}

We review some of notations of toric varieties.
Detailed discussion on the general theory can be found in \cite{CLS11}.
Let $\mathfrak M$ be a lattice of rank $n$, where $\mathfrak N=\Hom(\mathfrak M,\mathbb{Z})$ is
the $\mathbb{Z}$-dual of $\mathfrak M$. 
We define $\mathfrak M_\R:=\mathfrak M\otimes_\Z \R \cong \R^n$ 
(resp. $\mathfrak N_\R:=\mathfrak N\otimes_\Z \R$).
Let $\Sigma$ denote a complete fan in $\mathfrak N_{\R}$, i.e. $ \cup_{\sigma \in \Sigma}\sigma =\mathfrak {N}_{\R}$. 
The $k$-dimensional cones of $\Sigma$ form a set $\Sigma(k)$.
Let $\sigma$ be a cone in $\Sigma$.
The associated affine scheme $U_{\sigma}:={\mathrm{Spec}}\: \C[\mathfrak M\cap \sigma^{\vee}]$ 
is called an affine toric variety.
Then $\Sigma$ defines a toric variety $X:=X(\mathfrak N,\Sigma)$ by constructing 
the disjoint union of the affine toric varieties
$U_{\sigma}$, where one glues $U_{{\sigma}_1}$ and $U_{{\sigma}_2}$ along the open subvariety 
$U_{{\sigma}_1\cap {\sigma}_2}$, for $\sigma_1, \sigma_2 \in \Sigma$.
We generally  define a {\emph{toric variety}} $X$ as a %irreducible 
complex algebraic normal variety containing a torus $T_{\C}^n=\mathrm{Spec}\: \C[\mathfrak M]$
as a Zariski open subset, such that, the action of $T_{\C}^n$ 
on itself extends to an algebraic action of $T_{\C}^n$ on $X$.

A polytope $\D\subseteq \mathfrak M_{\R}$ is called a lattice (resp. rational) polytope 
if all its vertices are in $\mathfrak M$ (resp. $\mathfrak M_{\Q}=\mathfrak M\otimes_\Z \Q$).
Let $\Delta \subseteq \mathfrak M_{\R}$ be a rational $n$-dimensional polytope with $0\in {\mathrm{Int}} \Delta$.
We define the dual polytope $\Delta^{\circ} \subseteq \mathfrak N_{\R}$ by
\[
\Delta^{\circ}:=\set{ a\in \mathfrak N_{\R}| \braket{a,b}\geqslant -1 \; \text{for all } b\in \Delta },
\]
which is also a rational $n$-dimensional polytope with $0\in {\mathrm{Int}} \Delta^{\circ}$.
We denote a face $F$ of $\Delta^{\circ}$ by $F\prec \Delta^{\circ}$.
The fan $\mathcal{N}_{\Delta}:=\set{{\mathrm{pos}}(F)| F \prec \Delta^{\circ}}$ 
is called the {\emph{normal fan}} of $\Delta$,
where ${\mathrm{pos}}(F)$ is the linear positive hull of $F$. 
For a rational polytope $\Delta \subseteq \mathfrak M_{\R}$, we define
the associated toric variety by
$X_{\D}:=X(\mathfrak N,\mathcal{N}_{\D})$. In particular, when $\D$ satisfies 
Delzant's condition, it corresponds to a smooth compact toric manifold.
It is well-known that there is the bijective correspondence between $n$-dimensional
lattice polytopes and compact toric varieties with  
$(\C^{*})^n$-equivariant very ample line bundles. 

\vskip 8pt 
The discussion on examples (Section $\ref{sec:example}$) will focus on toric Fano threefolds.  We recall the related notations here.  See \cite{Deb03} and \cite{NiThesis} for more details.

Let $X$ be a complex projective normal variety. We call $X$ a {\emph{$\Q$-Fano variety}} if the anti-canonical divisor $-K_X$
is an ample $\Q$-Cartier divisor. Furthermore, $X$ is {\emph{Fano variety}} if $-K_X$ is an ample
Cartier divisor.
Let $P\subseteq \mathfrak N_{\R}$ be an $n$-dimensional lattice polytope which contains the origin in its interior.
Then $P$ is called a {\emph{canonical Fano polytope}} if $\mathrm{Int} P\cap \mathfrak N=\set{0}$.
Furthermore, $P$ is called a {\emph{Fano polytope}} if the vertices of any facet of $P$ form a 
$\Z$-basis of the lattice $\mathfrak N$.
By \cite[Proposition $2.3.7$]{NiThesis}, there is a bijective correspondence between isomorphism classes of Fano polytopes (resp. canonical Fano polytopes)
and smooth toric Fano varieties (resp. toric Fano varieties with canonical singularities).
Here and hereafter we assume that all singularities are at worst orbifold singularities.
 Hence they are associated to complete simplicial fans \cite[Theorem $3.1.19$]{CLS11}.

Let $X$ be a complex normal variety. Recall that $X$ is called 
{\emph{$\Q$-factorial}} if any Weil divisor is $\Q$-Cartier.
For the toric case, we have a well-known description in terms of a Fano polytope. A polytope is called {\emph{simplicial}}
if any facet is a simplex.
It was shown that simplicial Fano polytopes correspond uniquely up to 
isomorphism to $\Q$-factorial toric Fano varieties (see Proposition $2.3.12$ in \cite{NiThesis}). 
Since a toric variety $X$ has only finite quotient singularities (that is, $X$ is an orbifold) if and only if
 the associated fan $\Sigma$
is simplicial, $\Q$-factorial toric Fano varieties are toric Fano orbifolds.

In \cite{Kasp10}, Kasprzyk found that there are $12,190$ $\Q$-factorial toric Fano varieties up to isomorphism.
In the case when $X$ is smooth, toric Fano threefolds are classified in \cite{B81, WW82}. 
There are $18$ toric Fano threefolds, 
that is, 
$\C P^3$, $\mathcal{B}_1$, $\mathcal{B}_2$, $\mathcal{B}_3$, $\mathcal{B}_4=\C P^2\times \C P^1$,
$\mathcal{C}_1$, $\mathcal{C}_2$, $\mathcal{C}_3=\C P^1 \times \C P^1 \times \C P^1$,
$\mathcal{C}_4=\C P^1\times (\C P^2\; \# \;  \overline{\C P^2})$, 
$\mathcal{C}_5=\C P_{\C P^1\times \C P^1}(\mathcal{O} \oplus \mathcal{O} (1,-1))$, 
 $\mathcal{D}_1$, $\mathcal{D}_2$,
$\mathcal{E}_1$, $\mathcal{E}_2$, $\mathcal{E}_3=\C P^1\times (\C P^2\; \# \;  2\overline{\C P^2})$, $\mathcal{E}_4$,  $\mathcal{F}_1=\C P^1\times (\C P^2\; \# \; 3 \overline{\C P^2})$ and $\mathcal{F}_2$.
Note that we use the same notation as in \cite{B98}. 
These classification results are available online at \cite{graded_ring, Ob07}.

%subsec 2.2

\subsection{Reduction of relative $K$-stability}\label{redK}

In this subsection, we recall the reduction of the Futaki invariant on toric manifolds.
We also present the formulae to determine  the normalized potential $\theta_\D$ of the extremal vector field in symplectic coordinate and criterions for relative $K$-stability and instablity.

First we recall the Futaki invariant and the extremal vector field.  
Let $(X, \omega_g)$ be a compact K\"ahler manifold.
Let $\text{Aut}^0(X)$ be the identity component of the biholomorphisms
group of $X$. Then  $\text{Aut}^0(X)$ has a semidirect decomposition 
$$\text{Aut}^0(X)=\text{Aut}_r(X)\ltimes R_u, $$ 
where $\text{Aut}_r(X)$ is a reductive algebraic subgroup of
$\text{Aut}^0(X)$, which is the complexification
of a maximal compact subgroup $K$, and $R_u$ is the nilpotent radical of $\text{Aut}^0(X)$.
We denote
the Lie algebra of $\text{Aut}_r(X)$ by $\eta_r(X)$. Let $v\in\eta_r(X)$ so
that its imaginary part generates a one-parameter compact subgroup
of $K$. Then if the  K\"ahler form $\omega_g$ is $K$-invariant,
there exists a unique real-valued function $\theta_v(\omega_g)$ (called 
normalized potential of $v$) such that
\begin{equation}\label{norm1}
 i_v\omega_g =\sqrt{-1}\bar\partial \theta_v(\omega_g) \ \ \text{and}\ 
\int_X\theta_v(\omega_g)\frac{\omega_g^n}{n!}=0.
\end{equation}
For simplicity, we denote the set of such potentials $\theta_v$  by $ \Xi_{\omega_g}$.
Then the Futaki invariant on $\eta_r$ can be written as
\begin{equation}
F(v)=-\int_X \theta_v(\omega_g) (S(\omega_g)-\bar S)\frac{\omega_g^n}{n!},
\end{equation}
where $S(\omega_g)$ is the scalar curvature of $\omega_g$.
In \cite{FM95},  Futaki and Mabuchi defined the {\it extremal vector field}, $V=g^{i\bar j}(\text{proj} (S(\omega_g)))_{\bar j}\frac{\partial}{\partial z_i}$ in $\eta_r(X)$ for the K\"ahler
class $[\omega_g]$, where $\text{proj}(S(\omega_g))$ is the $L^2$-inner projection of 
the scalar curvature of $\omega_g$ to $ \Xi_{\omega_g}$.   
They showed that $V$ is
independent of the choice of $K$-invariant metrics in $[\omega_g]$,
and its potential is uniquely determined 
as the dual of the Futaki invariant with respect to the $L^2$ bilinear form
\begin{equation}
F(v)=
-\int_X \theta_v(\omega_g)\theta_V(\omega_g)\frac{\omega_g^n}{n!}, 
\ \ \forall\ v\in\eta_r(X).
\end{equation} 

Now we consider the reduction on a polarized toric manifolds $(X, L)$.
Choose an $(S^1)^n$-invariant K\"ahler metric $g$ with $\omega_g\in 2\pi c_1(L)$.
By choosing a base point, we identify
the open dense orbit of the complex torus action on $X$ with $(\mathbb C^{*})^n$
and use the coordinates $(z_1,\ldots,z_n)$ induced from $(\mathbb C^{*})^n$.
Denote the affine logarithmic coordinates
$w_i=\log z_i=y_i+\sqrt{-1}\eta_i$.
Then $\omega_g$ is determined by a smooth convex function $\varphi$
which depends only on $y_1,\ldots,y_n\in{\mathbb R^n}$
in the coordinates $(w_1,\ldots, w_n)$, namely
\begin{equation}\label{metric}
\omega_g =2\sqrt{-1}\partial\bar{\partial}\varphi
\end{equation} 
on $(\mathbb C^{*})^n$. As is well-known, the moment map can be given by $D \varphi$
and the image %by 
$\Delta =D\varphi(\mathbb R^n)$ is a polytope.
We denote by $x_i=\frac{\partial\varphi}{\partial y_i}$, $i=1,\ldots, n$, the symplectic coordinates.
Note that in the affine logarithm coordinates $(w_1,\ldots,w_n)$, 
$\{\frac{\partial}{\partial w_i},\ i=1,\ldots,n\}$ is a basis 
of the Lie algebra of $T_{\mathbb C}^n$ as a complex subalgebra of the 
Lie algebra of holomorphic vector fields on $X$. 
In particular, the Futaki invariant is given by the formula
\[
F(\frac{\p}{\p w_i})=\int_{\p \D} x_i\, d\sigma - \bar S \int_\D x_i \,dx
\]
and this leads the fact that $\mathcal L_\D (u)=0$ for any affine linear function $u$ as mentioned in
$\eqref{linearfunc}$.
The following lemma was given in \cite{ZZ08} on how to determine 
$\theta_V(\omega_g)$ in symplectic coordinates through the Futaki invariant.

\begin{lemma} \label{lemma:extremal}
Let $g$ be an $(S^1)^n$-invariant metric on $X$. 
Assume $V$ is the extremal vector field and
$\theta_V(\omega_g)$ is  
the normalized potential 
associated to $\omega_g$ by (\ref{norm1}). 
Then there are $2n$-numbers $a_i$ and $c_i$ such that
 $$\theta_V(\omega_g)=\sum_{i=1}^n a_i(x_i + c_i)=:\theta_\D,$$
  where $x=(x_i)=D\varphi\in \Delta$. Moreover $a_i$ and $c_i$ are determined uniquely by $2n$-equations,
\begin{eqnarray}
\frac{1}{(2\pi)^n}F(\frac{\partial}{\partial
w_i}) & =&- \int_\Delta \left(\sum_{j=1}^n a_j(x_j + c_j)\right)(x_i+c_i)
dx,\ i=1,\ldots,n,\\  
 \int_\Delta (x_i+c_i) dx &=& 0,\ i=1,\ldots,n.
\end{eqnarray}
\end{lemma}

As mentioned in the introduction, the relative $K$-stability in the toric sense
refers to the positivity of the linear functional (\ref{linearfunc})
for convex functions. 
Note that $\mathcal L_\D(u)$ is invariant when adding an affine linear function to $u$.
Without loss of generality, we assume $\mathbf 0$ lies in the interior of $\D$.
Hence, it suffices to consider convex functions  normalized at ${\mathbf 0}$ in the sense that 
$\inf_{x\in\Delta} u(x)=u({\mathbf 0})=0$. 
When $X$ is a toric Fano manifold, it is observed in \cite{ZZ08} that
 \begin{equation}\label{linearfunc2}
\mathcal L_\D(u)=\int_{\Delta}\left(\sum_{i=1}^n x_i u_i-u\right)+(1-\theta_{\D})u\,dx
\end{equation}
for $C^1$ functions by an integration by parts from $\eqref{linearfunc}$. 
Here $u_i=\frac{\partial u}{\partial x_i}$. By approximation,
it is easy to see that  $\eqref{linearfunc2}$ can also be used for the computation of $\mathcal L_\D(u)$ 
for piecewise $C^1$ functions.
As can be seen from $\eqref{linearfunc2}$, 
the positivity of $\mathcal L_\D$ relies heavily on the positivity of $1-\theta_\D$. 
Assume $\theta_\D=\sum_{i=1}^na_ix_i+c$.
Then we shall prove Theorem $\ref{prop:unstablecond}$.

\begin{proof}[Proof of Theorem $\ref{prop:unstablecond}$]
(1) is obvious \cite{ZZ08}. We only need to consider (2).
If $1-\theta_\D<0$, i.e. $\Delta^-=\Delta$, it is obvious that 
all simple piecewise linear convex functions of the form
$\max\{\sum_{i=1}^n b_ix_i, 0\}$
will destabilize $\Delta$. So we assume $1-\theta_\D=0$ intersects the interior of $\Delta$.
Let 
$$u=\max\{-(1-\theta_{\D}), 0\}.$$
Then $$\sum_{i=1}^n x_iu_i-u=\begin{cases}
1-c,  & x\in \Delta^-;\\
0 ,&  x \in \Delta\setminus \Delta^-.
\end{cases}$$
Hence, 
$$\mathcal L_\D(u)=(1-c){\mathrm{Vol}}(\Delta^{-})-\int_{\Delta^{-}} (1-\theta_\D)^2\,dx.$$
The theorem follows.
\end{proof}
The condition in this theorem is not sharp but we will see in
Section \ref{sec:example} that this criterion can determine all the stable toric Fano threefolds.

%sec3

\section{Relative Chow stability of toric manifolds}\label{sec:Chowsta}

In this section, we consider relative Chow stability of polarized toric manifolds.

%subsec 3.1

\subsection{Notions of Chow stabilities}\label{sec:lowest}
We first recall various notions of Chow stabilities.
We refer to the monograph \cite{Fut12} by Futaki for a more general concept of 
Chow stability in K\"ahler geometry. 
In \cite{Mab04, Mab14}, Mabuchi defined the notion of relative Chow stability in order to 
consider the existence problem of extremal K\"ahler metrics.
A historical background of relative GIT stability 
is given by Sz\'ekelyhidi \cite[Chapter $1$]{SzThesis}.

Let $G$ be a connected reductive complex algebraic group with Lie algebra $\mathfrak{g}$.
Let $\mathbf V$ be a finite dimensional complex vector space.
Suppose that $G$ acts linearly on ${\mathbf{V}}$.
Assume $v^{\ast}$ is  a nonzero vector in ${\mathbf{V}}$ which is a representative of 
$v=[v^{\ast}]\in \P({\mathbf{V}})$.
According to GIT,
$v^{\ast}$ is  {\it $G$-semistable} 
if the closure of the $G$-orbit $\mathcal{O}_G(v^{\ast})$ does not contain the origin. Furthermore,
$v^{\ast}$ is {\it $G$-stable} if $\mathcal{O}_G(v^{\ast})$ is closed.
We call $v^{\ast}$ {\it $G$-unstable} if $v^{\ast}$ is not $G$-semistable. Analogously, $v\in \P({\mathbf{V}})$ 
is said to be  $G$-semistable (resp. stable, unstable) 
if any representative of $v$ is $G$-semistable (resp. stable, unstable). 

To feature relative stability,  following \cite[Chapter $1$]{SzThesis}, we consider a torus $T$ in $G$, 
and denote its Lie algebra by $\mathfrak{t}$.
Then we define subalgebras
of $\mathfrak{g}$ by
\begin{eqnarray*}
\mathfrak{g}_T&=&\set{\alpha \in \mathfrak{g} | [\alpha, \beta]=0 \text{ for all } \beta\in \mathfrak{t}},\\
\mathfrak{g}_{T^\perp}&=&\set{\alpha \in \mathfrak{g}_T | \langle\alpha, \beta\rangle=0 \text{ for all } \beta\in \mathfrak{t}},
\end{eqnarray*}
where $\langle\cdot, \cdot\rangle$ is a rational invariant inner product.
We denote the image of $\mathfrak{g}_T$ (resp. $\mathfrak{g}_{T^\perp}$) 
under the exponential map by $G_T$ (resp. $G_{T^\perp}$).
\begin{definition}\cite{SzThesis}\rm \
Let $T$ be a torus in $G$ fixing the point $v$. 
Then $v$ is said to be \emph{semistable (resp. stable, unstable) relative to $T$} if
it is $G_{T^\perp}$-semistable (resp. stable, unstable).
\end{definition}

The Hilbert-Mumford criterion says that $v \in \mathbb P(\mathbf{V})$ is $G$-semistable if and only if  
$v$ is $H$-semistable for  
any maximal algebraic torus $H\subset G$ {\cite[p.$137$]{Dolga03}}.
When $G$ itself is isomorphic to an algebraic torus, the above stabilities can be described by the weight polytopes of the actions as follows.
Let $\chi(G)$ denote the character group of $G$. Then $\chi(G)$ consists of algebraic homomorphisms
$\chi: G\longrightarrow \C^{*}$. If we fix an isomorphism $G\cong (\C^{*})^{N+1}$, we may
express each $\chi$ as a Laurent monomial
\[
\chi(t_{1},\dots, t_{N+1})=t_{1}^{a_{1}}\cdots t_{N+1}^{a_{N+1}},\quad 
t_i \in \C^{*}, \; a_i\in \mathbb{Z}.
\]
Thus, there is the identification between $\chi(G)$ and $\mathbb{Z}^{N+1}$ by
$\chi=(a_{1},\dots,a_{N+1})\in \mathbb{Z}^{N+1}$.
Then it is well-known that ${\mathbf{V}}$ decomposes under the action of $G$ into weight spaces
\[
{\mathbf{V}}=\bigoplus_{\chi \in \chi(G)}{\mathbf{V}}_{\chi}, \qquad
{\mathbf{V}}_{\chi}:=\set{v^{\ast}\in {\mathbf{V}} | t\cdot v^{\ast} = \chi(t)\cdot
v^{\ast},\; t\in G}.
\]
\begin{definition}\rm
Let $v^{\ast}$ be a nonzero vector in ${\mathbf{V}}$ with $v^{\ast}=\sum_{\chi\in \chi(G)}v_{\chi}$, $v_{\chi}\in {\mathbf{V}}_{\chi}$.
The \emph{weight polytope} of $v^{\ast}$ (with respect to $G$-action) is the  
convex lattice polytope in $\chi(G)\otimes \R\cong \R^{N+1}$
defined by
\[
\mathcal{N}_G(v^{\ast}):={\mathrm{Conv}}\set{\chi\in \chi(G)|v_{\chi}\neq
0}\subseteq \R^{N+1}.
\]
\end{definition}

According to \cite[Theorem $9.2$]{Dolga03}, $v^{\ast}$ is $G$-semistable (resp. stable) if and only if $0\in\mathcal{N}_{G}(v^{\ast})$ (resp. $0\in\mathrm{Int}\: \mathcal{N}_{G}(v^{\ast})$).
In the relative stability setting, we also have the following.
\begin{proposition}\cite[Theorem $1.5.2$]{SzThesis}
Let $T$ be a torus in $G$ fixing the point $v$.
$v$ is semistable (resp. stable) relative to $T$ if and only if the orthogonal projection of the origin
onto the minimal affine subspace containing $\mathcal{N}_G(v^{\ast})$ is in $\mathcal{N}_G(v^{\ast})$
$(resp, {\mathrm{relint}}\, \mathcal{N}_G(v^{\ast}))$.
\end{proposition}

Next we define Chow form and Chow stability
of irreducible projective varieties.
See \cite{GKZ94,Yotsu15} for more details.
Let $X\subset \C P^N$ be an $n$-dimensional irreducible
complex projective variety of degree $d$. Recall that the
Grassmann variety $\mathbb{G}(k,\C P^N)$ parameterizes $k$-dimensional
projective linear subspaces of $\C P^N$.
The \emph{associated hypersurface} of $X\subset \C P^N$ is the
subvariety in $\mathbb{G}(N-n-1, \C P^N)$ which is given by
\[
Z_X:=\set{W \in \mathbb{G}(N-n-1,\C P^N)| W\cap X\neq \emptyset}.
\]
It is known that is $Z_X$ is an irreducible hypersurface with $\deg Z_X=d$ in the Pl\"ucker coordinates. In particular,
$Z_X$ is given by the vanishing of a section $R_X^{\ast}\in H^0(\mathbb{G}(N-n-1,\C P^N),\mathcal{O}(d))$.
We call $R_X^{\ast}$ the \emph{Chow form} of $X$. Note that
$R_X^{\ast}$ is well defined up to a multiplicative constant. Let
${\mathbf{V}}:=H^0(\mathbb{G}(N-n-1,\C P^N),\mathcal{O}(d))$ and $R_X \in \P({\mathbf{V}})$ be the
projectivization of $R_X^{\ast}$. We call $R_X$ the \emph{Chow point} of
$X$. The weight polytope of $R_X^{\ast} \in {\mathbf{V}}$ with respect to 
the action $(\C^{*})^{N+1}\subset  {\mathrm{GL}}(N+1,\C)$ of diagonal matrices 
is called \emph{Chow polytope} of $X$, 
and is denoted by ${\mathrm{Ch}}(X)$. See \cite[Chapter $6$]{GKZ94} for more details.
Since we have the natural action of $G={\rm{SL}}(N+1,\C)$ into
$\P({\mathbf{V}})$, we can define stabilities of $R_X$ as follows.
\begin{definition}\rm
Let $X\subset \C P^N$ be an irreducible, $n$-dimensional
complex projective variety. Then $X$ is said to be \emph{Chow
semistable (resp. stable, unstable)} if the Chow point $R_X$ of $X$ is
${\rm{SL}}(N+1,\C)$-semistable (resp. stable, unstable).
\end{definition}

We consider relative Chow stability 
when the Futaki invariant does not vanish.  
Choose $T={\beta}$ to be the $\mathbb C^*$- action  induced by  the extremal vector field $V$. 
$T$ also acts  on $\P(\mathbf{V})$.

\begin{definition}\label{def:relativeChowsta}\rm
Let $X\subset \C P^N$ be an irreducible, $n$-dimensional
complex projective variety. Then $X$ is said to be \emph{relatively Chow
semistable (resp. stable, unstable)} if the Chow point $R_X$ of $X$ is
$\mathrm{SL}(N+1,\C)$-semistable (resp. stable, unstable) relative to $T$.
\end{definition}

\begin{definition}\label{def:AsymrelativeChowsta}\rm
Let $(X,L)$ be a polarized variety. For $i\gg 0$,
let $\Psi_i:X\longrightarrow \P(H^0(X,L^i)^{\ast})$ be
the Kodaira embedding.
\begin{enumerate}
\item Suppose that $L$ is very ample. $(X,L)$ is said to be {\emph{relatively Chow semistable (resp. stable, unstable)}}
if $\Psi_{1}(X)\subset  \P(H^0(X,L)^{\ast})$ is relatively Chow semistable (resp. stable, unstable).

\item $(X,L)$ is called \emph{asymptotically relatively Chow semistable (resp. stable)} if there is an $i_0$ such that 
$\Psi_i(X)$ is relatively Chow 
semistable (resp. stable) for each $i \geqslant i_0$.
\end{enumerate}
We say that $(X,L)$ is {\emph{asymptotically relatively Chow unstable}} if it is not asymptotically relatively 
Chow semistable.
\end{definition}

%subsec 3.2

\subsection{Reduction on toric manifolds}\label{chow-reduction}

We reduce the relative Chow stability of polarized toric manifolds by developing on an idea 
in \cite{Ono11, Ono13}.

Recall that 
$T_{\D}^{\C}$  of $\mathrm{SL}(N+1,\C)$ is given by
\begin{align*}
T_{\D}^{\C} &\hookrightarrow G=(\C^{*})^{N+1}\cap \mathrm{SL}(N+1,\C)\\
(t_1,\dots,t_{N})&\longmapsto (t_1,\dots,t_{N},(t_1\cdots t_{N})^{-1}).
\end{align*}
In particular, $T_{\D}^{\C}\cong (\C^{*})^{N}$.
We view the Lie algebra of $T^\C_{\D}$ as a subalgebra of $sl(N+1,\C)$
by considering the traceless part. The inner product $\langle ,\rangle$ on $sl(N+1,\C)$
is given by $\langle A, B\rangle=Tr(AB)$.
Let $\set{\mathbf{a}_1,\dots, \mathbf{a}_{N+1}}$ be all the lattice points in $\D$.
We define
\[
\bar \theta_{\D} =\frac{1}{N+1}\displaystyle\sum_{j=1}^{N+1}\theta_\D(\mathbf{a}_j).
\]
Let $\theta_\D$ be the potential function  as
in Section $\ref{sec:prelim}$, and  
$T=\beta$ be the $\mathbb C^*$- action  induced by $V$. Then $T$ is given in $T_{\D}^{\C}$ by
\begin{align}\label{map:1-psExtK}
\begin{split}
T:\C^{*}&\hookrightarrow G \\
t\;&\longmapsto \left(t^{(\theta_\D({\mathbf{a}_1})-\bar{\theta}_{\D})},\dots,t^{(\theta_\D(\mathbf{a}_{N+1})-\bar{\theta}_{\D})} \right).
\end{split}
\end{align}

Let $\mathrm{Ch}(\D)$ be the Chow polytope of $X_{\D}\subset \C P^{N}$. 
In the literature of Gelfand-Kapranov-Zelevinsky's theory, $\mathrm{Ch}(\D)$
coincides with the secondary polytope \cite{KSZ92}. 
In particular, it is known that the affine span of the secondary polytope
is given by the following.
 
\begin{proposition}\cite[Chapter 7, Proposition $1.11$]{GKZ94}\label{prop:GKZ}
Let $\varphi=(\varphi_1,\dots,\varphi_{N+1})$ be a point in the affine hull of $\mathrm{Ch}(\D)$ in 
$\chi (G)\otimes \R\cong \R^{N+1}$.
Then 
\begin{eqnarray}
\sum_{j=1}^{N+1}\varphi_j=(n+1)!\mathrm{Vol}(\D), \label{eq:secpoly1}\ \ \ 
\sum_{j=1}^{N+1}\varphi_j\mathbf{a}_j=(n+1)!\int_{\D}{\mathbf{x}}\, dx.
\end{eqnarray}
Here ${\mathbf{x}}=(x_1,\dots ,x_n)$ 
and $\D \cap M=\set{{\mathbf{a}}_1,\dots,{\mathbf{a}}_{N+1}}$ is all the lattice points in $\D$.
\end{proposition}

Denote  
$$d_{\D}=(1,\dots,1),\ \ \tilde{\theta}_{\D}=((\theta_\D\left(\mathbf{a}_1\right)-\bar \theta_{\D}),\dots,(\theta_\D\left(\mathbf{a}_{N+1}\right)-\bar \theta_{\D}))$$ in $\chi (G)\otimes \R$.  Then we have the following.

\begin{theorem}\label{th:relative chow polytope} 
$(X_\D, L_\D)$ is relatively Chow semistable in the toric sense if and only if
 there exists $s\in \R$ such that
\begin{equation}\label{eq:relChowss}
\frac{(n+1)!\mathrm{Vol}(\D)}{N+1}\left(d_{\D}+s\tilde{\theta}_{\D}\right)\in \mathrm{Ch}(\D).
\end{equation}
Furthermore, it is relatively Chow stable in the toric sense if 
\begin{equation}\label{eq:relChows}
\frac{(n+1)!\mathrm{Vol}(\D)}{N+1}\left(d_{\D}+s\tilde{\theta}_{\D}\right)\in \mathrm{Int}(\mathrm{Ch}(\D)).
\end{equation}
\end{theorem}

\begin{proof}
We define a two dimensional subspace in $\R^{N+1}$ by
$\mathbf{W}:=\mathrm{Span}_{\R}\set{d_{\D},\tilde{\theta}_{\D}}$.
Let $\beta_1,\dots, \beta_{N-1}\in \R^{N+1}$ be 
a basis of the subspace perpendicular to ${\mathbf{W}}$.
Note that $G_{T^{\perp}}$ is isomorphic to $(\C^{*})^{N-1}$.
Considering the projection
\begin{align*}
\pi_{G_{T^{\perp}}}: \chi(G)\otimes \R\cong \R^{N+1}& \longrightarrow \chi(G_{T^{\perp}})\otimes \R \cong \R^{N-1} \\
\varphi=(\varphi_1,\dots ,\varphi_{N+1}) &\longmapsto (\bracket{\varphi, \beta_1} , \dots ,\bracket{\varphi, \beta_{N-1}} ),
\end{align*}
we observe that  $\mathcal{N}_{G_{T^{\perp}}}(R_{X_\D})=\pi_{G_{T^{\perp}}}(\mathcal{N}_G(R_{X_\D}))\subset \R^{N-1}$.

By definition, 
$R_{X_\D}$ is $G_{T^\perp}$-semistable if and only if
${0}\in \mathcal{N}_{G_{T^{\perp}}}(R_{X_{\D}})$. By the projection above, it is equivalent to $\mathbf{W}\cap\mathrm{Ch}(\D)\neq \emptyset$,
that is, there exist $s_1,s_2\in \R$ such that

\begin{equation}\label{eq:relChowss1}
s_1d_{\D}+s_2\tilde{\theta}_{\D}\in \mathrm{Ch}(\D).
\end{equation}
By $\eqref{eq:secpoly1}$ and the fact $\sum_{j=1}^{N+1}\tilde{\theta}_{\D}(\mathbf{a}_j)=0$, 
we have 
$s_1=\frac{(n+1)!{\mathrm{Vol}}(\D)}{N+1}$.
\end{proof}
\begin{remark}\rm

The reader should notice that $\eqref{eq:relChowss}$ implies 
\begin{equation}\label{eq:secpoly22}
\sum_{j=1}^{N+1}\mathbf{a}_j+s\sum_{j=1}^{N+1}
\tilde{\theta}_\D({\mathbf{a}})
\mathbf{a}_j=\frac{N+1}{\mathrm{Vol}(\D)}\int_\D \mathbf{x}\, dx,
\end{equation}  by Proposition $\ref{prop:GKZ}$.
This theorem extends Ono's description of Chow semistability \cite{Ono11, Ono13} to relative case.
\end{remark}

Next, we consider asymptotic relative Chow semistability.
Denote the Ehrhart polynomial of $\D$ by $E_{\D}(t)$.
It has degree $n=\dim \D$ and satisfies
\[
E_{\D}(i)=\mathrm{Card} (i\D \cap \Z^n)=\mathrm{Card} (\D \cap (\Z/i)^n)
\] 
for 
any positive integer $i\in \Z_+$. Moreover,  
$E_{\D}(t)$ has the form
\[
E_{\D}(t)=\mathrm{Vol} (\D)t^n+\frac{\mathrm{Vol} (\p \D)}{2}t^{n-1}+\dots +1,
\]
by theorem of Ehrhart.
Note that $\displaystyle \int_{i\D}\mathbf x\, dx=i^{n+1}\int_\D\mathbf x\, dx$, 
hence Theorem \ref{def:ARCSStoric} follows 
from the same argument in the proof of Theorem $\ref{th:relative chow polytope}$.

The asymptotic relative Chow semistability can be related to relative $K$-semistability in the toric sense. 
For this purpose, we recall some notations. 
For a fixed $\varphi\in W(i\D)$(see $\eqref{eq:SetOfCharFun}$ for the definition of $W(i\D)$), let 
\[
G_\varphi=\text{the convex hull of \  } \bigcup_{\mathbf{a}\in \D\cap (\Z/i)^n} \{(\mathbf{a}, t) \ |  \
t \leqslant \varphi(\mathbf{a})\}\subset \mathfrak M_\R\times \R \cong \R^{n+1}.
\]
Then we define a 
piecewise linear function $g_{\varphi}: \D \longrightarrow \R$ by
\[
g_\varphi(\mathbf{x})=\max\{ t\  |\  (\mathbf{x}, t)\in G_\varphi\}.
\]
The upper boundary of $G_{\varphi}$ can be regarded as the graph of $g_{\varphi}$. 
Furthermore, $g_{\varphi}$ has the following properties. 
\begin{lemma}[\cite{GKZ94} p.$221$, Lemma $1.9$]\label{lem:gkz}
For any $\varphi \in W(i\D)$,
\begin{enumerate}
\item[(a)] the function $g_{\varphi}$ is concave.
\item[(b)] we have the equality
\[
\max\{ \langle\varphi, \psi\rangle\ |\ \psi\in \mathrm{Ch}(i\D)\}=i^n(n+1)! \int_\D g_\varphi\, dx.
\]
\end{enumerate}
\end{lemma}

Denote
$PL(\D, i)=\{ g_\varphi\ | \ \varphi\in W(i\D)\}.$
The proof of Theorem \ref{rel-pl} is similar to \cite{Ono13}.
One can see that the condition $\eqref{assy-rel-chow}$ holds if and only if the following condition holds:
\begin{equation}\label{eq:assy-rel-chow2}
\max\{ \langle\varphi, \psi\rangle\ |\ \psi\in \mathrm{Ch}(i\D)\}\geqslant \frac{i^n(n+1)!\mathrm{Vol}(\D)}{E_{\D}(i)}\langle\varphi, 1+s_i\tilde\theta_{i\D}\rangle
\end{equation}
for all $\varphi\in W(i\D)$.
By applying Lemma $\ref{lem:gkz}$ to $\eqref{eq:assy-rel-chow2}$, we obtain $\eqref{eq:funcQ}$.

\begin{corollary}\label{cor:necAsymprelChowss}
If $(X_\D, L_\D)$ is asymptotically relatively Chow semistable,
then for any $i\in \Z_+$, there exists $s_i$ such that 
\begin{equation}\label{eq:secpoly2-2}
\sum_{\mathbf{a}\in \D\cap (\Z/i)^n}i\mathbf{a}+s_i\sum_{\mathbf{a}\in \D\cap (\Z/i)^n}
(\theta_\D(\mathbf{a})-\bar{\theta}_{i\D})
\tilde{\theta}_{i\D}(\mathbf{a})\mathbf{a}=\frac{iE_{\D}(i)}{\mathrm{Vol}(\D)}\int_\D \mathbf{x}\, dx.
\end{equation}
\end{corollary}

\begin{proof}
Since $(X_\D, L_\D)$ is asymptotically relatively Chow semistable, $(X_\D, L_\D^i)$ is relatively Chow semistable
for $i \gg 0$. By Theorem $\ref{rel-pl}$, $\mathcal Q_\D(i,g)\geqslant 0$ for all $g\in PL(\D, i)$.
Taking $g$ to affine linear functions, one can see that $\mathcal Q_\D(i,g)=0$.
Note that $\mathcal Q_\D(i,g)$ can be written as a fractional polynomial in $i$ as in $\eqref{eq:funcQ-Asymp}$.
We prove the assertion by contradiction. If $\eqref{eq:secpoly2-2}$ does not hold for an integer $i_0\in \Z_+$,
(i.e. $\mathcal Q_\D(i_0, \mathbf{x})\neq 0 $), then the identity theorem for polynomial functions implies that
there is an integer $i_1\in \Z_+$ such that $\mathcal Q_\D(i, \mathbf{x})\neq 0$ for any $i> i_1$.
This means $(X_\D, L_\D^i)$ is relatively Chow unstable for $i>i_1$, that is, $(X_\D, L_\D)$ is asymptotically
relatively Chow unstable. The assertion is proved.
%Applying $\eqref{eq:funcQ}$ to affine linear functions, one can see that $\mathcal Q_\D(i,g)=0$.
%As in $\eqref{eq:funcQ-Asymp}$, $\mathcal Q_\D(i,g)$ can be written as a fractional polynomial in $i$. 
%If $\eqref{eq:secpoly2-2}$ does not hold for an integer 
%$i_0\in \Z_+$, then there is an integer $i_1\in \Z_+$ such that $(X_\D,L_\D^i)$ is relatively Chow unstable for any $i > i_1$.
\end{proof}

\begin{remark}\label{rem:necAsympChow}\rm
(1) One can see that if $s_i$ exists, then \eqref{s-explicit} follows from $\eqref{eq:secpoly2-2}$.
(2) It is clear that $\eqref{eq:secpoly2-2}$ is an over-determined linear system since 
there is only one parameter $s_i$, but $n$ equations. Hence, one can expect to
find counter-examples from polytopes which are not symmetric with respect to $x_1,\ldots, x_n$. (Cf. Proposition $\ref{prop:relChowunst3}$).
(3) When $V=0$, $\eqref{eq:secpoly2-2}$ becomes
\begin{equation}\label{eq:necChowss}
\mathrm{Vol}(\D)s_\D(i)-E_\D(i)\int_\D{\mathbf{x}}\;dx=\mathbf 0
\end{equation}
which is the necessary condition for $(X_\D, L_\D^i)$ to be Chow semistable proved by Ono\cite[Theorem $1.4$]{Ono11}. 
From the argument in the proof of Corollary $\ref{cor:necAsymprelChowss}$, one can see that the following:
let $\D \subseteq \mathfrak M_\R$ be a simple lattice polytope and $(X_\D, L_\D)$ be the associated polarized toric orbifold. 
If $\eqref{eq:necChowss}$ does not hold for a positive integer $i_0\in \Z_+$, then there is a positive integer $i_1\in \Z_+$, such that
$(X_\D,L_\D^i)$ is Chow unstable for any $i>i_1$.
\end{remark}

%sec 3.4

%\subsection{Asymptotic relative Chow semistability implies relative $K$-semistability}

Finally we show asymptotical relative Chow semistability implies relative $K$-semistability in the toric case. 
It is known that asymptotic Chow semistability implies $K$-semistability in general sense\cite{RT07}. 
First, we need the following lemma.
\begin{lemma}\label{lem:secpoly}
$s_i$ in (\ref{eq:secpoly2-2}) satisfies
$s_i=-\frac{1}{2}+O(i^{-1})$.
\end{lemma}

\begin{proof}
By $\displaystyle \int_\D \theta_\D\, dx=0$
and Lemma $3.3$ of \cite{ZZ08}, we have
\begin{align*}
\bar{\theta}_{i\D}&= \frac{i^{n-1}}{2E_\D(i)}\int_{\p\D}\theta_\D d\sigma+O(i^{n-2}),\\[3pt]
\sum_{{\bf a}\in \D\cap (\Z / i)^n} i\mathbf a&=i^{n+1}\int_\D \mathbf{x}\, dx+\frac{i^{n}}{2}\int_{\p \D}\mathbf x\, d\sigma
+O(i^{n-1}),\qquad \text{and}\\
\sum_{{\bf a}\in \D\cap (\Z/i)^n} \theta_{\D}(\mathbf a) \mathbf a&= i^{n}\int_\D \theta_\D \mathbf x\, dx +O(i^{n-1}).
\end{align*}
Then $\eqref{eq:secpoly2-2}$ is written as
\begin{eqnarray*}
 &&{\hspace{-1.5cm}{iE_\D(i)\int_\D \mathbf x\, dx
=\mathrm{Vol}(\D)\left(i^{n+1}\int_\D \mathbf x\, dx+\frac{i^n}{2} \int_{\p \D} \mathbf x\, d\sigma
 +O(i^{n-1})\right)}} + s_i\mathrm{Vol}(\D)\cdot\\[3pt]
&&{\hspace{-0.5cm}{\left[i^{n}\int_\D \theta_\D \mathbf x\, dx+O(i^{n-1})-\frac{i^n\int_\D \mathbf x\, dx+O(i^{n-1})}{E_\D (i)}\left(\frac{i^{n-1}}{2}\int_{\p \D}\theta_\D\, d\sigma+O(i^{n-2})\right)\right].}} 
\end{eqnarray*}
Comparing the coefficient of $i^n$, we conclude that
\[
\mathrm{Vol}(\D)\left[\frac{1}{2} \int_{\p \D} \mathbf x\, d\sigma+
s_i\int_\D \theta_\D \mathbf x\, dx\right]= \frac{1}{2} \mathrm{Vol}(\p \D)\int_{\D} \mathbf x\, dx.
\]
Since $\theta_\D$ is the potential function of the extremal vector field, it holds
\[
\int_{\p \D}x_k\, d\sigma-\int_\D \left(\frac{\mathrm{Vol}(\p \D)}{\mathrm{Vol}(\D)}+\theta_\D \right)x_k\,dx=0 \quad \text{for} \quad
 k=1,\ldots, n.
\]
Hence we have $s_i=-\frac{1}{2}+O(i^{-1})$. 
\end{proof}

Now we prove Theorem $\ref{thm:relChowKss}$.
\begin{proof}[Proof of Theorem $\ref{thm:relChowKss}$]
For any $i\in\Z_+$ and $g\in PL(\D,i)$,
by Lemma \ref{lem:secpoly},
\begin{align*}
\begin{split}
\mathcal Q_\D(i, g)=& E_\D(i)\int_\D g\, dx -\mathrm{Vol}(\D) \sum_{\mathbf{a}\in \D \cap (\Z / i)^n} \left(-\frac{\theta_{\D}(\mathbf{a})-\bar\theta_{i\D}}{2i}+1+O(i^{-2})\right)g(\mathbf{a})\\
=&\left(\mathrm{Vol}(\D)i^{n}+\frac{\mathrm{Vol}(\p \D)}{2}i^{n-1}+O(i^{n-2})\right)\int_\D g\, dx 
+i^{n-1}\mathrm{Vol}(\D) 
\int_\D \frac{\theta_{\D}-\bar\theta_{i\D}}{2}g\, dx \\
&+O(i^{n-2})-\left(i^{n}\mathrm{Vol}(\D)\int_\D g\, dx+\frac{i^{n-1}{\mathrm{Vol}(\D)}}{2}\int_{\partial \D}g\, d\sigma+O(i^{n-2})\right).
\end{split}
\end{align*}
Note that $\bar{\theta}_{i\D}=O(i^{-1})$, then

\begin{equation} \label{eq:funcQ-Asymp}
\begin{split}
\mathcal Q_\D(i, g)
=&-\frac{\mathrm{Vol}(\D)}{2}\left[\int_{\p \D}g\, d\sigma-\int_\D \left(\frac{\mathrm{Vol}(\p \D)}{\mathrm{Vol}(\D)}+\theta_{\D}\right)g\, dx\right] i^{n-1}+O(i^{n-2}) \\
=&-\frac{\mathrm{Vol}(\D)}{2}\mathcal L_\D(g) i^{n-1}+O(i^{n-2}).
\end{split}
\end{equation}

By the assumption, $\mathcal Q_\D(i, g)\geq 0$ when $i$ is sufficiently large.
This implies $\mathcal L_\D(g)\leq 0$. Hence, $\mathcal L_\D(u)\geq 0$ for any piecewise linear convex function $u$. The theorem is proved.
\end{proof}

%sec 4

\section{Reduction of relative Chow stability: an alternative approach}\label{redu-2}

The asymptotic Chow stability can also be described through the Chow weight of $\C^*$-actions induced by test configurations  \cite{D02, RT07}.  In this section, we derive an alternative reduction of relative Chow stability of toric manifolds by investigating the normalized weight  in \cite{RT07} for toric degenerations.

First, we recall the notion of a test configuration \cite{D02}.
%
%\begin{definition}\rm
A {\it test configuration} for a polarized scheme $(X, L)$ is a polarized scheme $(\mathfrak X, \mathfrak L)$
with:
\begin{itemize}
\item a $\C^*$-action and a proper flat morphism $\pi: \mathfrak X\to \C$ which is $\C^*$-equivariant
for the usual action on $\C$,
\item a $\C^*$-equivariant line bundle $\mathfrak L\to \mathfrak X$ which is ample over all fibers of 
$\pi$ such that for $z\neq 0$, $(X, L)$ is isomorphic to $(\mathfrak X_z, \mathfrak L_z)$, $\mathfrak L_z=\mathfrak L|_{\mathfrak X_z}$.
\end{itemize}
$(\mathfrak X, \mathfrak L)$ is called {\it product} if $\mathfrak X=X\times \C$, 
and {\it trivial} if in addition
the $\C^*$ acts only on $\C$.
%\end{definition}
\vskip 8pt

It is shown in \cite{RT07} that, for any $i\in \Z_+$, the data of a test configuration for $(X, L^i)$ gives 
a $\mathbb C^*$- action
of $GL(d_i,\C)$ and vice versa, where $d_i=\dim H^0(X, L^i)$. Let $(\mathfrak X_0, \mathfrak L_0)$ be
the central fiber. For any $r\in \Z_+$,
let $k=ri$. Let 
$w(k)=w(ri)$  be the total weight of the induced $\C^*$-action on 
$H^0(\mathfrak X_0, \mathfrak L_0^r)$.
As in \cite{RT07}, we define the normalized weight $\tilde w_{i,k}$ by 
\begin{equation}\label{norm-weight}
\tilde w_{i, k}=w(k)id_i-w(i)kd_k.
\end{equation}
By general algebraic theory, $\tilde w_{i, k}$ is a polynomial of  degree $n+1$ in $k$, for $k \gg0$.
Write 
\[\tilde w_{i, k}=\displaystyle\sum_{j=1}^{n+1}e_{j}(i)k^j.
\] 
Then the leading coefficient $e_{n+1}(i)$ %term $e_{n+1}(i)i^{n+1}(n+1)!$ 
is the {\it Chow weight}.  
Then using the Hibert-Mumford criterion for the $\C^*$-actions, Chow stability is described as follows.
\begin{theorem}{(\cite{RT07})}
A polarized variety $(X, L)$ is Chow stable with respect to $i$ if for any nontrivial test configuration for $(X, L^i)$, $e_{n+1}(i)>0$.  $(X, L)$ is asymptotically Chow stable if there exists $i_0$ such that for $i\geqslant i_0$, any nontrivial test configuration for $(X, L^i)$ has $e_{n+1}(i)>0$. 
\end{theorem}

Now we consider toric manifolds. Recall that a piecewise linear convex function $u=\mathrm{max}\set{f_1,\ldots ,f_{\ell}}$ is called {\it rational} if $f_k=\sum a_{k,i} x_i + c_k, \ k=1,\ldots,\ell$,
for some vectors $(a_{k,1},\ldots, a_{k,n})\in \mathbb R^n$ and some numbers
$c_k\in \mathbb R$ such that all $a_{k, i}$ and $c_k$ are rational. 
According to  \cite{D02}, a {\it toric degeneration} for $(X_\D, L_\D^i)$ is a test configuration
induced by a rational piecewise  linear convex function $u$ on $\D$,
such that $iQ$ is a lattice polytope in $\R^{n+1}=\R^n\times \R$.
Here $R$ is an integer such that $u\leqslant R$ and
$$Q=\set{({\mathbf{x}},t) |  {\mathbf{x}}\in \D,\  0<t<R-u({\mathbf{x}})}.$$
Denote the set of all rational piecewise  linear convex functions satisfying the above condition 
by $PL_i(\D)_\Q$.
\begin{lemma}\label{chow-weight-toric}
The Chow weight of the toric degeneration for $(X_\D, L_\D^i)$ induced by $u$
is given by  $e_{n+1}(i)=-i\mathcal P_\D(i, u)$, where
\begin{eqnarray}\label{chow-weight}
\mathcal P_\D(i, u)=E_\D(i)\int_\D u \,dx -\mathrm{Vol} (\D) \sum_{\mathbf{a}\in \D \cap (\Z / i)^n} u(\mathbf{a}).
\end{eqnarray}
\end{lemma}
\begin{proof}
By computation \cite{D02, ZZ08}, 
\begin{eqnarray*}
d_i&=&E_\D(i),\\
d_k&=& E_\D(k)=\mathrm{Vol} (\D)k^n+\frac{\mathrm{Vol} (\p \D)}{2}k^{n-1}+\dots +1,\\
w(i)&=&i \sum_{\mathbf{a}\in \D \cap (\Z / i)^n} (R-u)(\mathbf{a}),\\
w(k)&=&k^{n+1} \int_\D (R-u) \,dx +\frac{k^n}{2}\int_{\partial\D} (R-u) \,d\sigma +\cdots.
\end{eqnarray*}
Substituting them into (\ref{norm-weight}), we get
\begin{eqnarray}
e_{n+1}(i)&=&iE_\D(i)\int_\D (R-u) \,dx -\mathrm{Vol} (\D) i\sum_{\mathbf{a}\in \D \cap (\Z / i)^n} (R-u)(\mathbf{a})\nonumber\\
&=& i\left[E_\D(i)\int_\D (-u) \,dx -\mathrm{Vol} (\D) \sum_{\mathbf{a}\in \D \cap (\Z / i)^n} (-u)(\mathbf{a})\right].\nonumber
\end{eqnarray}
\end{proof}
%which is independent of $R$. When $V=0$  one can easily see that  $e_{n+1}(i)$ is $i\mathcal Q_\D(i, -u)$.

More generally, for the purpose of relative stability, i.e. when $V\neq 0$, we consider the toric degenerations  perpendicular to the $\C^*$-action $\beta$ induced by $V$.
We use the notations of toric data in Section \ref{sec:Chowsta}.
Following \cite{Mab14}, we view the Lie algebra of $T^\C_{i\D}$ as a subalgebra of $sl(E_\D(i), \mathbb C)$
by considering the traceless part. The inner product $\langle ,\rangle_i$ on $sl(E_\D(i), \mathbb C)$
is given by 
\begin{equation}\label{inner-product}
\langle A, B\rangle_i=\frac{Tr(AB)}{i^{n+2}}.
\end{equation}
Now let $\alpha$ be the $\C^*$-action on the central fiber induced by the toric degeneration.
According to \cite{ZZ08}, the infinitesimal generators of $\alpha$ and $\beta$ are
\[
\mathrm{diag}(i(R-u)(\mathbf{a}_1), \ldots, i(R-u)(\mathbf{a}_{E_\D(i)})),\ \mathrm{diag}(i\theta_\D(\mathbf{a}_1), \ldots, i\theta_\D(\mathbf{a}_{E_\D(i)}))
\]
respectively. By considering the traceless parts of them with (\ref{inner-product}), we call the toric degeneration is {\it perpendicular} to $\beta$ if 
\begin{equation}\label{orth}
\sum_{\mathbf{a}\in \D \cap (\Z / i)^n} (R-u)(\mathbf{a}) (\theta_\D(\mathbf{a})-\bar\theta_{i\D})=\sum_{\mathbf{a}\in \D \cap (\Z / i)^n} u(\mathbf{a}) (\theta_\D(\mathbf{a})-\bar\theta_{i\D})=0.
\end{equation}
In view of (\ref{chow-weight}), we have:

\begin{definition}\rm\label{definition:relK}
$(X_\Delta, L_\D)$ is called {\it  asymptotically relatively Chow semistable for toric degenerations} 
if there exists $i_0\in \Z_+$, such that when $i\geqslant i_0$,  the Chow weight of
any toric degeneration for $(X_\D, L_\D^i)$ which is perpendicular to $\beta$  is nonnegative. 
 Furthermore, it is called {\it asymptotically relatively Chow stable for toric degenerations} if
the Chow weight is positive for any nontrivial toric degeneration for $(X_\D, L_\D^i)$.
\end{definition}

\begin{proposition}\label{weakchow}
$(X_\Delta, L_\D)$ is asymptotically relatively Chow semistable for toric degenerations if and only if $\eqref{eq:secpoly2-2}$ holds and there exists $i_0\in \Z_+$, such that when $i\geqslant i_0$, for any $u\in PL_i(\D)_\Q$, $\mathcal Q_\D(i, u)\leqslant 0$.
\end{proposition}

\begin{proof}
We consider the projection of $u$ onto the perpendicular space. 
Let
\[
\tilde u=u-\frac{\displaystyle\sum_{\mathbf{a}\in \D \cap (\Z / i)^n} u(\mathbf{a}) (\theta_\D(\mathbf{a})-\bar\theta_{i\D})}{\displaystyle\sum_{\mathbf{a}\in \D \cap (\Z / i)^n}  (\theta_\D(\mathbf{a})-\bar\theta_{i\D})^2} (\theta_\D-\bar\theta_{i\D}).
\] 
Then there exist $r\in \Z_+$ such that $r\tilde u$ induces a toric degeneration for $(X_\D, L_\D^i)$ perpendicular to $\beta$. 
By (\ref{chow-weight}), we have
\begin{eqnarray*}
\mathcal P_\D(i, r\tilde u)
&=&E_\D(i)\int_\D r\tilde u \,dx -\mathrm{Vol} (\D) \displaystyle\sum_{\mathbf{a}\in \D \cap (\Z / i)^n} r\tilde u (\mathbf{a})\\
&=& E_\D(i)\int_\D ru\,dx- \mathrm{Vol} (\D)  \sum_{\mathbf{a}\in \D \cap (\Z / i)^n}  (1+s_i\tilde\theta_{i\D}(\mathbf{a}))ru(\mathbf{a})
=r\mathcal Q_\D(i, u),
\end{eqnarray*}
where  $s_i$ is given by \eqref{s-explicit},
%\begin{equation}\label{s-explicit}
%s_i=\frac{i\bar\theta_{i\D}E_\D(i)}{\displaystyle\sum_{\mathbf{a}\in \D \cap (\Z / i)^n}  (\theta_\D(\mathbf{a})-\bar\theta_{i\D})^2}.
%\end{equation}
and the normalization condition $\int_\D \theta_\D\,dx=0$ is used. 
\end{proof}

\begin{remark}\rm
The sign differs from \eqref{eq:funcQ} because we consider convex functions here, while in \eqref{eq:funcQ} $g$ is a concave function.  Futhermore, one can see that if the nonpositivity of $\mathcal Q_\D(i, \cdot)$ is strengthened for all function in $PL(\D, i)$, it corresponds to the asymptotic relative Chow semistability in the toric sense. The necessary condition  $\eqref{eq:secpoly2-2}$ for asymptotic relative Chow semistability in the 
toric sense in the last section can also be recovered by
substituting $u=x_j$ and $-x_j$ into the above for $j=1,\ldots, n$.
%It is also easy to see that in $\eqref{eq:secpoly2-2}$, if the $s_i$ exists, it can be given by
%$\eqref{s-explicit}$.
\end{remark}

%sec 5

\section{Examples}\label{sec:example}

Finally we provide many interesting examples 
which will support our understanding of stabilities. 
Here we will mainly concentrate on toric Fano manifolds.

%sec 5.1

\subsection{Relative $K$-stability of toric Fano threefolds}\label{sec:computation}

In this section, we shall determine the potential $\theta_\D$ of 
the extremal vector field $V$ of toric Fano threefolds in symplectic coordinates 
and verify the relative $K$-stability or instability by Theorem $\ref{prop:unstablecond}$.

If $V$ is given by 
$V=\sum_{i=1}^na_i\frac{\p}{\p w_i}$ in the
affine logarithm coordinates $(w_1,\dots,w_n)$, 
then the potential function $\theta_\D$  is given by 
$\theta_\D=\sum_{i=1}^n a_ix_i+c$
for some constant $c$. There are several ways to compute $\theta_\D$. The most general one is to 
use the linear functional $\eqref{linearfunc}$. 
In order to determine constants $a_i$ and $c$, one can
solve the $n+1$-linear system
\begin{equation}\label{eq:lin_fun}
\mathcal L_\D(1)=0, \; \mathcal L_\D(x_i)=0 \qquad \text{for} \qquad i=1,\dots,n.
\end{equation}
%This method works for general polarizations.

In Fano case, we have a more efficient algorithm.
By \cite{BS99}, $\varphi$ in (\ref{metric}) can be given by
\begin{equation}\label{varphiformula}
\varphi=\log\left(\sum_{i=1}^m e^{\langle p^{(i)}, y\rangle}\right),
\end{equation} 
where $p^{(1)}$,\ldots, $p^{(m)}$ are the vertices of $\Delta$. The Fano assumption implies
\begin{equation}\label{varphibound}
|\varphi+\log\det(\varphi_{ij})|<\infty.
\end{equation} 
We recall another normalization on the potentials of holomorphic vector fields.
Let $v$ be a holomorphic vector field on $X$ and $\theta'_v(\omega_g)$ 
be the potential function determined by
\begin{equation}\label{potential-normalization}
i_v\omega_g=\sqrt{-1}\bar\partial\theta'_v(\omega_g), \quad \text{and} \quad\ \int_X\theta'_v(\omega_g)e^{h_g}\frac{\omega_g^n}{n!} =0,
\end{equation}
where $h_g$ is a Ricci potential of $\omega_g$.
According to \cite{TZ02}, the above $\theta'_v$ satisfies 
\begin{equation}\label{futakiformula}
\theta'_v=-\Delta_g \theta'_v-v(h_g),
\end{equation}
where $\D_g$ is the Laplacian.  Assume
$v=\sum_{i=1}^n a_iz_i\frac{\partial}{\partial z_{i}}=\sum_{i=1}^n a_i\frac{\partial}{\partial w_{i}}$
on the torus orbit, where $a_i\in\mathbb R$. 
Then we have $\theta'_v(\omega_g)=\sum_{i=1}^na_{i}x_{i}$
in the symplectic coordinates by a simple observation from \eqref{varphiformula}. In particular, 
$\theta'_{\frac{\partial}{\partial w_i}}(\omega_g)=x_i$ for $i=1,\ldots, n$.
 This simple fact has been used in \cite{ZZ08}.
Then by \eqref{futakiformula}, 
we can compute the Futaki invariant  by
\begin{equation}\label{eq:Futaki2}
F(v)=\int_Xv(h_g)\omega^n=-\int_X\theta'_v\omega^n=-(2\pi)^2\int_{\Delta}\theta'_v dx.
\end{equation}
The first step of the algorithm is to determine $a_1$,\ldots, $a_n$ by (\ref{eq:Futaki2}) and  Lemma \ref{lemma:extremal}.
Then 
\begin{equation}\label{eq:const}
c=-\frac{1}{{\mathrm{Vol}}(\Delta)}\int_{\D}\displaystyle\sum_i a_ix_i\, dx
\end{equation}
by the normalization condition.
An alternative method to compute $a_1$,\ldots, $a_n$ was also given by Nakagawa. 
He gave a combinatorial formula for Futaki invariant and the generalized Killing form
of toric Fano orbifolds in \cite{Naka98}. 
As  an application, he computed the extremal vector field in 
the anti-canonical class on a toric Fano manifold $X$ with $\dim X \leqslant 4$.
In order to prove Theorem $\ref{thm:main}$, we shall use his result on  toric Fano threefolds 
directly.

The main goal of this section is to prove the following proposition.
\begin{proposition}\label{prop:toricFano}
Let $X$ be a toric Fano threefold with anti-canonical polarization.
\begin{enumerate}
\item[(a)] If $X=\mathcal{B}_2$, then $X$ is relatively (strongly) $K$-stable in the toric sense in the anti-canonical class.
\item[(b)] If $X=\mathcal{B}_1$, then $X$ is relatively $K$-unstable in its first Chern class.
\end{enumerate}
\end{proposition}
Once Proposition $\ref{prop:toricFano}$ has been proved,
other cases are similar and further details are left to the reader\footnote{In the practical computation 
we used packages (i) {\tt{Normaliz}} and (ii) {\tt{Polymake}}. 
These packages are available at \cite{normaliz} and \cite{polymake} 
respectively.}. In Table $\ref{list-stability}$ and Table $\ref{table:toricFano3}$ we give the list of  all results proved in Theorem $\ref{thm:main}$.

\begin{proof}
(a)
Let $e_i \;(i=1,2,3)$ be the standard basis of $\mathfrak N\cong \mathbb{Z}^3$.
Let $\Sigma$ be the complete fan in $\mathfrak N_{\R}\cong \R^3$ whose $1$-dimensional cones are given by
$\Sigma(1)=\set{\sigma_1,\sigma_2,\sigma_3,\sigma_4,\sigma_5}$ where
\begin{eqnarray*}
\sigma_1={\mathrm{Cone}}(e_1), \hspace{0.54cm}  \sigma_2={\mathrm{Cone}}(e_2), \hspace{0.54cm} \sigma_3={\mathrm{Cone}}(e_3), \\
\sigma_4={\mathrm{Cone}}(-e_3), \quad \text{and} \quad  \sigma_5={\mathrm{Cone}}(-e_1-e_2-e_3).
\end{eqnarray*}
Then the associated toric manifold $X$ is $\C P(\mathcal{O}_{\C P^2}\oplus \mathcal{O}_{\C P^2}(1))$
and 
\begin{align*}
\D=\left \{(x_1,x_2,x_3)\in \R^3 \;| \; x_1\geqslant -1, \; x_2\geqslant -1, \; x_3\geqslant -1, \; x_3 \leqslant 1, x_1+x_2+x_3\leqslant 1 \; \right \}. 
\end{align*}
Hence ${\mathrm{Vol}}(\Delta)=\frac{28}{3}$. Let $V\in \eta_c(X)$ be the extremal vector field in the anti-canonical class and $\theta_\D$
be the potential function of $V$. Then 
$\theta_\D=-\frac{70}{97}x_3+c$ for some constant $c$. (See \cite{Naka98}, Section $6$, Table $2$).
Since we have $\displaystyle \int_{\D}x_3=-2$, we conclude that
$c=-\frac{15}{97}$
by $\eqref{eq:const}$. Thus $\theta_\D=-\frac{70}{97}x_3-\frac{15}{97}$ and
\[
\D^-=\set{(x_1,x_2,x_3)\in \D | -\frac{112}{97}-\frac{70}{97}x_3\geqslant 0 }=\emptyset.
\]
$X$ satisfies the condition $\eqref{ineq:normal_pot}$.
The assertion is verified. 

\noindent (b) Let $\Sigma$ be the complete fan in $N_{\R}\cong \R^3$ 
whose $1$-dimensional cones are given by
$\Sigma(1)=\set{\sigma_1, \dots, \sigma_5}$ where
\begin{eqnarray*}
\sigma_1={\mathrm{Cone}}(e_1), \hspace{0.63cm}  \sigma_2={\mathrm{Cone}}(e_2), \hspace{0.63cm} 
\sigma_3={\mathrm{Cone}}(e_3), \\
\sigma_4={\mathrm{Cone}}(-e_3), \quad \text{and} \quad  \sigma_5={\mathrm{Cone}}(-e_1-e_2-2e_3).
\end{eqnarray*}
Then we readily see that $X=\C P(\mathcal{O}_{\C P^2}\oplus \mathcal{O}_{\C P^2}(2))$ and 
\begin{eqnarray*}
\D={\mathrm{Conv}}\set{ e_3^*-e_2^*,\; 4e_1^*-e_2^*-e_3^*,\; -e_1^*-e_2^*-e_3^*,\; e_3^*-e_1^*-e_2^*,\; 4e_2^*-e_1^*-e_3^*,\; e_3^*-e_1^*},
\end{eqnarray*}
where $e_i^*$ is the dual basis of $e_i$. Then ${\mathrm{Vol}}(\Delta)=\frac{31}{3}$ and
$\theta_\D=-\frac{620}{349}x_3+c$
for some constant $c$. Since $\displaystyle \int_{\D}x_3=-4$, we obtain $c=-\frac{240}{349}$.
Hence we conclude $\theta_\D=-\frac{620}{349}x_3-\frac{240}{349}$ and
\begin{align*}
\D^-&=\set{(x_1,x_2,x_3)\in \D | -\frac{589}{349}-\frac{620}{349}x_3\geqslant 0 } \\
&={\mathrm{Conv}}\left\{
\begin{matrix}
4e_1^*-e_2^*-e_3^*,& \frac{39}{10}e_1^*-e_2^*-\frac{19}{20}e_3^*,& -e_1^*-e_2^*-\frac{19}{20}e_3^* \\
&& \\
-e_1^*+\frac{39}{10}e_2^*-\frac{19}{20}e_3^*,& -e_1^*-e_2^*-e_3^*,& -e_1^*+4e_2^*-e_3^*
\end{matrix}
\right \}.
\end{align*}
Thus 
\begin{equation}\label{eq:vol}
{\mathrm{Vol}}(\Delta^{-})=\frac{7351}{12000}.
\end{equation}

Now we shall verify Condition $\eqref{unstablecond}$. First we note that
$\int_{\D^-}x_3\, dx = -\frac{96197}{4}$ and $\int_{\D^-}x_3^2\, dx = \frac{1828273}{5}$.
Hence one can see that
\[
\int_{\D^-}(1-\theta_\D)^2\, dx=\int_{\D^-}\left(\frac{589}{349}+\frac{620}{349}x_3 \right)^2\, dx=\frac{1475918766336271}{1461612000}.
\]
Plugging this and $\eqref{eq:vol}$ into $\eqref{unstablecond}$, we obtain the desired result because $1-c=\frac{589}{349}$.
\end{proof}

%\begin{remark}\rm
%It is also interesting to ask if the method in Section $\ref{sec:computation}$ can determine all toric Fano fourfolds. However, we have some technical difficulties on computing \eqref{unstablecond}. 
%That is, the vertices of $\D^-$ can be very complicated which leads 
%to the overflow of the software {\tt Normaliz} when we compute
% the term $\int_{\D^-}(1-\theta_\D)^2 dx$ while some of them are still computable. 
%Our practical computation shows that the stability of all computable ones can be determined
%by  Theorem $\ref{prop:unstablecond}$.
%That is among all toric Fano fourfolds ($124$ deformation classes), 
%there are $52$ types of relative $K$-stable ones, 
%$15$ types of relative $K$-unstable ones, and 
%$57$ types of undistinguishable ones. 
%We expect that the $57$ types of undistinguishable ones can be settled by using more advanced software or improvement of the criterion in  Thereom $\ref{prop:unstablecond}$.
%\end{remark}

%sec 5.2

\subsection{Relative Chow stability}\label{counter-chow}

We study relatively Chow unstable examples of toric Fano manifolds.
First, we recall the example found by Nill and Paffenholz \cite{NP11} which is isomorphic to 
$\P(\mathcal{O}_W\oplus \mathcal{O}_W(-1,-1,-1,2))=:X_{\mathrm{NP}}$ where $W=(\C P^1)^3\times \C P^3$.
$X_{\mathrm{NP}}$ is a non-symmetric K\"ahler-Einstein toric Fano $7$-fold. In \cite{OSY12}, 
Ono, Sano and Yotsutani 
showed that $(X_{\mathrm{NP}}, -K_{X_{\mathrm{NP}}})$ is asymptotically Chow unstable. Later Yotsutani observed that
$X_{\mathrm{NP}}$ is Chow unstable w.r.t. $-K_{X_{\mathrm{NP}}}$ (i.e. $i=1$) 
using $\eqref{eq:necChowss}$. 
 Meanwhile, Nill and Paffenholz \cite{NP11} also proved that 
all toric K\"ahler-Einstein Fano manifolds are symmetric if $\dim X\leqslant 6$.  
It is known that if $X$ is a symmetric toric Fano manifold, then $(X, -K_X)$ is asymptotically Chow stable.
Hence,
the lowest dimension for an anti-canonically polarized K\"ahler-Einstein toric Fano manifold $(X, -K_X)$ to be
(asymptotically) Chow unstable is $7$.
However, in the case where $X$ is a Fano orbifold, 
such an example appears in $\dim X=3$.
\begin{proposition}\label{thm:Chowunst}
There is a K\"ahler-Einstein toric Fano orbifold $X$ with $\dim X=3$ which is Chow unstable w.r.t. $-2K_X$.
\end{proposition}

Our strategy is the following. For any toric Fano orbifold $X$ with the associated simplicial Fano polytope 
$\D^{\circ}\subseteq \mathfrak N_{\R}$, $X$ admits  K\"ahler-Einstein metric if and only if the Futaki invariant of $X$ vanishes
by Shi-Zhu's result \cite{SZ12}. Let $\mathcal{W}(X)$ be the Weyl group of ${\mathrm{Aut}}(X)$ with respect to the
maximal torus and $\mathfrak N_{\R}^{\mathcal{W}(X)}$ be the $\mathcal{W}(X)$-invariant subspace of $\mathfrak N_{\R}$.
It is known that $\sum_{v\in \mathcal{V}(\D^{\circ})}v\in \mathfrak N_{\R}^{\mathcal{W}(X)}$ (see \cite[Chapter $5$]{NiThesis}).
Here $\mathcal{V}(\D^{\circ})$ denotes the set of vertices of $\D^{\circ}$. Note that 
$X$ is symmetric if and only if $\mathfrak N_{\R}^{\mathcal{W}(X)}=\set{\mathbf{0}}$.
Hence we may consider toric Fano orbifolds with the vanishing Futaki invariant satisfying 
$\sum_{v\in \mathcal{V}(\D^{\circ})}v\neq \mathbf{0}$.
Among all $12,190$ $3$-dimensional toric Fano orbifolds (i.e. $\Q$-factorial toric Fano varieties),
there are $42$ toric Fano orbifolds with the vanishing Futaki invariant. Of these $42$ toric Fano orbifolds,
there is the only one example satisfying the above conditions.
Then it suffices to check 
$\eqref{eq:necChowss}$ for the dual moment polytope
$\D \subseteq \mathfrak M_{\R}$ of this one.
Note that the Gorenstein index $j_X$ is given by minimal $k$ such that $k\D$ is a lattice polytope for a fixed canonical
Fano polytope $\D^{\circ}\subseteq \mathfrak N_{\R}$ \cite[Proposition $2.3.2$]{NiThesis}. 
\begin{proof}[Proof of Proposition $\ref{thm:Chowunst}$]
Again we use the same notations as in the proof of Proposition $\ref{prop:toricFano}$.
We consider the $3$-dimensional canonical Fano polytope\footnote{ID number in the database \cite{graded_ring}
is $530571$.}
\[
\D^{\circ}:={\mathrm{Conv}}\set{e_1-e_2-2e_3, e_2+3e_3, e_1+e_2+3e_3, e_1+2e_2+4e_3, e_2, 
-2e_1-2e_2-3e_3}.
\]
The vertices of the dual polytope $\D$ are
\[
 -\frac{1}{2}e_1^*+\frac{5}{2}e_2^*-e_3^*, e_1^*-e_2^*, 2e_2^*-e_3^*, \frac{1}{2}e_2^*-\frac{1}{2}e_3^*,
-e_1^*, -e_1^*-e_2^*+\frac{1}{2}e_3^*, \frac{3}{2}e_1^*-e_2^*, -e_2^*+e_3^* .
\]
Thus $\D$ is 
a simple polytope and $j_{X_{\D}}=2$.
Setting $\widetilde{\D}:=2\D$, we compute the Chow weight of $\widetilde{\D}$.
We readily see that
\begin{equation*}
E_{\widetilde{\D}}(i)=12i^3+9i^2+3i+1, \quad {\mathrm{Vol}}(\widetilde{\D})=12, \quad 
\int_{\widetilde{\D}}{\mathbf{x}}\, dx =(0,0,0),\quad \int_{\p \widetilde{\D}}{\mathbf{x}} 
\,d\sigma=(0,0,0)
\end{equation*}
and
\begin{equation*}
\sum_{{\mathbf{a}}\in \widetilde{\D}\cap \Z^3} {\mathbf{a}}=(0,1,-1), \quad 
\sum_{{\mathbf{a}}\in 2\widetilde{\D}\cap \Z^3} {\mathbf{a}}=(0,3,-3), \quad 
\sum_{{\mathbf{a}}\in 3\widetilde{\D}\cap \Z^3} {\mathbf{a}}=(0,6,-6)
\end{equation*}
hold. 
This implies that
\[
\int_{\p \widetilde{\D}}{\mathbf{x}}\, d\sigma -
\frac{{\mathrm{Vol}}(\p \widetilde{\D})}{{\mathrm{Vol}}( \widetilde{\D})}\int_{\widetilde{\D}}{\mathbf{x}}\, dx=(0,0,0)
\]
i.e. the Futaki invariant vanishes.
However obviously,
\[
\sum_{\mathbf{a}\in \widetilde{\D}\cap \Z^3}{\mathbf{a}}\neq \frac{E_{\widetilde{\D}}(1)}{{\mathrm{Vol}}(\widetilde{\D})}
\int_{\widetilde{\D}} \mathbf{x}\, dx.
\]
By \eqref{eq:necChowss}, the $2$-Gorenstein toric Fano variety $(X_{\D}, -2K_{X_{\D}})$ is Chow unstable.
\end{proof}
In \cite{Ber12}, Berman proved that a $\Q$-Fano variety $X$ admitting K\"ahler-Einstein metrics is 
$K$-stable. Hence Remark $\ref{rem:necAsympChow}$ (2) gives the following.
\begin{corollary}
The example in 
Proposition $\ref{thm:Chowunst}$ is $K$-stable but asymptotically Chow unstable.
\end{corollary} 

Next, we consider the general case with the nontrivial Futaki invariant. 
We see that a smooth counter-example appears in $\dim X=3$. 
\begin{proposition}\label{prop:relChowunst3}
Let $X$ be a toric Fano threefold which is isomorphic to $\mathcal{E}_4$.
Then $(X,-K_X)$ is relatively $K$-stable but it is asymptotically relatively Chow unstable.
\end{proposition}
\begin{proof}
It suffices to see that $\eqref{eq:secpoly2-2}$ is not satisfied.
The corresponding $3$-dimensional moment polytope is listed in Table $\ref{table:MomentPoly}$. %$\ref{table:toricFanoRelSta}$.
Thus we have
\[
E_{\D}(i)=\frac{20}{3}i^3+10i^2+\frac{16}{3}i+1 \quad \text{and} \quad
\int_{\D}{\mathbf{x}}\, dx=\left(-\frac{7}{8}, \frac{5}{12}, \frac{5}{24}\right).
\]
In particular,
\[
\sum_{\mathbf{a}\in\D\cap \Z^3}{\mathbf{a}}=(-4,2,1) \quad \text{and} \quad 
s_1\sum_{\mathbf{a}\in\D\cap \Z^3} (\theta_\D(\mathbf{a})-\bar{\theta}_\D)\mathbf{a}=s_1
\left(-\frac{11134272}{1816885}, \frac{1079424}{363377}, \frac{539712}{363377} \right)
\]
holds. Thus there is no $s_1$ satisfying $\eqref{eq:secpoly2-2}$.
The assertion follows from Corollary $\ref{cor:necAsymprelChowss}$.
\end{proof}

%%%%%%%%%%%%%%%%%%%%%Table2%%%%%%%%%%%%%%%%%%%%%%%%%%%%

%\newpage

\begin{table}[H]
\caption{$\theta_\D$ and $\D^-$ of toric Fano threefolds}\label{table:toricFano3}
\begin{center}
\begin{tabular}{ccc}
\toprule
 Notation  & $\theta_\D=\sum_{i=1}^3a_ix_i+c$ & $\D^-$ \\ \midrule
 $\C P^3$  & $\equiv 0$ & $\emptyset$  \\[6pt]
 $\mathcal{B}_1$  & $ -\frac{620}{349}x_3-\frac{240}{349}$ & 
$\small{
\mathrm{Conv}\left\{
\begin{pmatrix}
4 \\
-1 \\
-1 \\
\end{pmatrix}, 
\begin{pmatrix}
\frac{39}{10} \\
-1 \\
-\frac{19}{20} \\
\end{pmatrix}, 
\begin{pmatrix}
-1 \\
-1 \\
-\frac{19}{20} \\
\end{pmatrix}, 
\begin{pmatrix}
 -1 \\
\frac{39}{10} \\
-\frac{19}{20} \\
\end{pmatrix},
 \right.}$
 \\ 
&& \ \ \ \ \ \ \ \ \ \ \ \
$\small{\left. \hspace{2cm} 
\begin{pmatrix}
-1 \\
-1 \\
-1 \\
\end{pmatrix},
\begin{pmatrix}
-1 \\
4 \\
-1 \\
\end{pmatrix}
\right\}}$ \\ [10pt]
$\mathcal{B}_2$  & $-\frac{70}{97}x_3-\frac{15}{97}$ & $\emptyset$ \\ [3pt]
$\mathcal{B}_3$  & $-\frac{20}{43}x_1-\frac{20}{43}x_2-\frac{5}{43}$ & $\emptyset$ \\ [3pt]
$\mathcal{B}_4$  & $\equiv 0$ & $\emptyset$  \\ [3pt]
 $\mathcal{C}_1$ & $-\frac{260}{219}x_3-\frac{80}{219}$ & $\emptyset$ \\ 
[6pt]
$\mathcal{C}_2$  & $-\frac{7600}{17787}x_1-\frac{17750}{17787}x_3-\frac{4868}{17787}$ &
$\small{
\mathrm{Conv}\left\{
\begin{pmatrix}
-\frac{981}{1520} \\
-1 \\
-1 \\
\end{pmatrix}, 
\begin{pmatrix}
-\frac{981}{1520} \\
\frac{4021}{1520} \\
-1 \\
\end{pmatrix}, 
\begin{pmatrix}
-1 \\
\frac{10111}{3550} \\
-\frac{3011}{3550} \\
\end{pmatrix}, 
 \right.}$
 \\ 
&&
$\small{\left. \hspace{2cm}
\begin{pmatrix}
 -1 \\
3 \\
-1 \\
\end{pmatrix}, 
\begin{pmatrix}
-1 \\
-1 \\
-\frac{3011}{3550} \\
\end{pmatrix},
\begin{pmatrix}
-1 \\
-1 \\
-1 \\
\end{pmatrix}
\right\}}$ \\ 
[6pt]
 $\mathcal{C}_3$  & $\equiv 0$ & $\emptyset$  \\[3pt]
$\mathcal{C}_4$  & $-\frac{6}{11}x_2-\frac{1}{11}$ & $\emptyset$ \\[3pt]
 $\mathcal{C}_5$  & $\equiv 0$ & $\emptyset$  \\
[6pt]
$\mathcal{D}_1$  & $ \frac{99600}{467581}x_1-\frac{627000}{467581}x_2-\frac{213939}{467581}$ & 
$\small{
\mathrm{Conv}\left\{
\begin{pmatrix}
\frac{48388}{18165} \vspace{0.1cm}\\
-\frac{12058}{18165} \\
-1 \\
\end{pmatrix}, 
\begin{pmatrix}
\frac{1363}{2490} \\
-1 \\
\frac{3617}{2490} \\
\end{pmatrix}, 
\begin{pmatrix}
3 \\
-1 \\
-1 \\
\end{pmatrix}, 
\begin{pmatrix}
\frac{1363}{2490} \\
-1 \\
-1 \\
\end{pmatrix}
\right\}}$
 \\ 
&& \\
$\mathcal{D}_2$  & $ \frac{219420}{650251}x_1-\frac{318320}{650251}x_2-\frac{62565}{650251}$ & 
$\small{
\mathrm{Conv}\left\{
\begin{pmatrix}
2\\
-\frac{1489}{1730} \\
-1 \\
\end{pmatrix}, 
\begin{pmatrix}
\frac{4288}{2385} \\
-1 \\
-\frac{1903}{2385} \\
\end{pmatrix}, 
\begin{pmatrix}
2 \\
-1 \\
-1 \\
\end{pmatrix}, 
\begin{pmatrix}
\frac{4288}{2385} \\
-1 \\
-1 \\
\end{pmatrix}
\right\}}$ \\ 
&& \\
$\mathcal{E}_1$  & $ -\frac{17020}{19651}x_1-\frac{17020}{19651}x_2-\frac{6845}{19651}$ & 
$\small{
\mathrm{Conv}\left\{
\begin{pmatrix}
-\frac{103}{185}\\
-1 \\
-1 \\
\end{pmatrix}, 
\begin{pmatrix}
-\frac{103}{185} \\
-1 \\
-\frac{473}{185} \\
\end{pmatrix}, 
\begin{pmatrix}
-1 \\
-\frac{103}{185} \\
-1 \\
\end{pmatrix}, 
 \right.}$
 \\ 
&&
$\small{\left. \hspace{2cm}
\begin{pmatrix}
 -1 \\
-\frac{103}{185} \\
\frac{473}{185} \\
\end{pmatrix}, 
\begin{pmatrix}
-1 \\
-1 \\
3 \\
\end{pmatrix},
\begin{pmatrix}
-1 \\
-1 \\
-1 \\
\end{pmatrix}
\right\}}$ \\ [5pt]
%$\mathcal{E}_2$  & $ -\frac{2646160}{2735927}x_1-\frac{982960}{2735927}x_2-\frac{692905}{2735927}$ & 
%$\small{
%\mathrm{Conv}\left\{
%\begin{pmatrix}
%-\frac{13897}{15035}\\
%-1 \\
%-1 \\
%\end{pmatrix}, 
%\begin{pmatrix}
%-\frac{13897}{15035} \\
%-1 \\
%-\frac{28932}{15035} \\
%\end{pmatrix}, 
%\begin{pmatrix}
%-1 \\
%-\frac{4447}{5585} \\
%-1 \\
%\end{pmatrix}, 
 %\right.}$
 %\\ 
%&&
%$\small{\left. \hspace{2cm}
%\begin{pmatrix}
% -1 \\
%-\frac{4447}{5585} \\
%2 \\
%\end{pmatrix}, 
%\begin{pmatrix}
%-1 \\
%-1 \\
%2 \\
%\end{pmatrix},
%\begin{pmatrix}
%-1 \\
%-1 \\
%-1 \\
%\end{pmatrix}
%\right\}}$ \\ 
\bottomrule
\end{tabular}
\end{center}
\end{table}

\begin{table}[H]
\begin{center}
\begin{tabular}{ccc}
\toprule
 Notation  & $\theta_\D$ & $\D^-$  \\ \midrule
&& \\
$\mathcal{E}_2$  & $ -\frac{2646160}{2735927}x_1-\frac{982960}{2735927}x_2-\frac{692905}{2735927}$ & 
$\small{
\mathrm{Conv}\left\{
\begin{pmatrix}
-\frac{13897}{15035}\\
-1 \\
-1 \\
\end{pmatrix}, 
\begin{pmatrix}
-\frac{13897}{15035} \\
-1 \\
-\frac{28932}{15035} \\
\end{pmatrix}, 
\begin{pmatrix}
-1 \\
-\frac{4447}{5585} \\
-1 \\
\end{pmatrix}, 
\right.}$
 \\ 
&&
$\small{\left. \hspace{2cm}
\begin{pmatrix}
 -1 \\
-\frac{4447}{5585} \\
2 \\
\end{pmatrix}, 
\begin{pmatrix}
-1 \\
-1 \\
2 \\
\end{pmatrix},
\begin{pmatrix}
-1 \\
-1 \\
-1 \\
\end{pmatrix}
\right\}}$ \\ 
 $\mathcal{E}_3$  & $-\frac{168}{409}x_1-\frac{168}{409}x_2-\frac{32}{409}$ & $\emptyset$ \\ 
[5pt]
 $\mathcal{E}_4$  & $-\frac{34208}{78995}x_1+\frac{7936}{78995}x_2-\frac{24929}{394975}$ & $\emptyset$ \\ 
[5pt]
 $\mathcal{F}_1$  & $\equiv 0$ & $\emptyset$  \\
[5pt]
 $\mathcal{F}_2$  & $\frac{36}{67}x_2-\frac{5}{67}$ & $\emptyset$ \\ 
[4pt]
\bottomrule
\end{tabular}
\end{center}
\end{table}

%%%%%%%%%%%%%%%%%%%%%Table3%%%%%%%%%%%%%%%%%%%%%%%%%%%%

%\newpage

\begin{table}[H]
\caption{The moment polytopes of toric Fano threefolds}\label{table:MomentPoly}
\begin{center}
\begin{tabular}{ccc}
\toprule
Notation  & The moment polytope $\D$  \\ \midrule
&&\\
 $\C P^3$  & $\small{ \mathrm{Conv}\left\{
\begin{pmatrix}
-1 \\
-1 \\
-1 \\
\end{pmatrix}, 
\begin{pmatrix}
-1 \\
-1 \\
3 \\
\end{pmatrix}, 
\begin{pmatrix}
-1 \\
3 \\
-1 \\
\end{pmatrix}, 
\begin{pmatrix}
3 \\
-1 \\
-1 \\
\end{pmatrix}\right\}}$ \\ 
&&\\
$\mathcal{B}_1$  & $\small{\mathrm{Conv}\left\{
\begin{pmatrix}
0 \\
-1 \\
1 \\
\end{pmatrix}, 
\begin{pmatrix}
4 \\
-1 \\
-1 \\
\end{pmatrix}, 
\begin{pmatrix}
-1 \\
-1 \\
-1 \\
\end{pmatrix}, 
\begin{pmatrix}
-1\\
-1 \\
1 \\
\end{pmatrix},
\begin{pmatrix}
-1\\
4 \\
-1 \\
\end{pmatrix},
\begin{pmatrix}
-1\\
0 \\
1 \\
\end{pmatrix}\right\}}$ \\
&&\\
 $\mathcal{B}_2$  & $\small{\mathrm{Conv}\left\{
\begin{pmatrix}
1 \\
-1 \\
1 \\
\end{pmatrix}, 
\begin{pmatrix}
3 \\
-1 \\
-1 \\
\end{pmatrix}, 
\begin{pmatrix}
-1 \\
-1 \\
-1 \\
\end{pmatrix}, 
\begin{pmatrix}
-1\\
-1 \\
1 \\
\end{pmatrix},
\begin{pmatrix}
-1\\
3 \\
-1 \\
\end{pmatrix},
\begin{pmatrix}
-1\\
1 \\
1 \\
\end{pmatrix}\right\}}$ \\
&&\\
 $\mathcal{B}_3$  & $\small{\mathrm{Conv}\left\{
\begin{pmatrix}
2 \\
-1 \\
0 \\
\end{pmatrix}, 
\begin{pmatrix}
2 \\
-1 \\
-1 \\
\end{pmatrix}, 
\begin{pmatrix}
-1 \\
2 \\
-1 \\
\end{pmatrix}, 
\begin{pmatrix}
-1\\
-1 \\
-1 \\
\end{pmatrix},
\begin{pmatrix}
-1\\
2 \\
0 \\
\end{pmatrix},
\begin{pmatrix}
-1\\
-1 \\
3 \\
\end{pmatrix}\right\}}$ \\
&&\\
 $\mathcal{B}_4$  & $\small{\mathrm{Conv}\left\{
\begin{pmatrix}
2 \\
-1 \\
1 \\
\end{pmatrix}, 
\begin{pmatrix}
2 \\
-1 \\
-1 \\
\end{pmatrix}, 
\begin{pmatrix}
-1 \\
-1 \\
-1 \\
\end{pmatrix}, 
\begin{pmatrix}
-1\\
-1 \\
1 \\
\end{pmatrix},
\begin{pmatrix}
-1\\
2 \\
-1 \\
\end{pmatrix},
\begin{pmatrix}
-1\\
2 \\
1 \\
\end{pmatrix} \right\}}$ \\
&&\\
 $\mathcal{C}_1$  & $\small{\mathrm{Conv}\left\{
\begin{pmatrix}
0 \\
-1 \\
1 \\
\end{pmatrix}, 
\begin{pmatrix}
2 \\
-1 \\
-1 \\
\end{pmatrix}, 
\begin{pmatrix}
2 \\
2 \\
-1 \\
\end{pmatrix}, 
\begin{pmatrix}
0\\
0 \\
1 \\
\end{pmatrix}, 
\begin{pmatrix}
-1\\
2 \\
-1 \\
\end{pmatrix},
\begin{pmatrix}
-1\\
0 \\
1 \\
\end{pmatrix},
\begin{pmatrix}
-1\\
-1 \\
-1 \\
\end{pmatrix},
\begin{pmatrix}
-1\\
-1 \\
1 \\
\end{pmatrix}\right\}}$ \\
&&\\
 $\mathcal{C}_2$  & $\small{\mathrm{Conv}\left\{
\begin{pmatrix}
0 \\
-1 \\
1 \\
\end{pmatrix}, 
\begin{pmatrix}
2 \\
-1 \\
-1 \\
\end{pmatrix}, 
\begin{pmatrix}
2 \\
0 \\
-1 \\
\end{pmatrix}, 
\begin{pmatrix}
0\\
0 \\
1 \\
\end{pmatrix}, 
\begin{pmatrix}
-1\\
3 \\
-1 \\
\end{pmatrix},
\begin{pmatrix}
-1\\
1 \\
1 \\
\end{pmatrix},
\begin{pmatrix}
-1\\
-1 \\
-1 \\
\end{pmatrix},
\begin{pmatrix}
-1\\
-1 \\
1 \\
\end{pmatrix}\right\}}$ \\
&&\\
$\mathcal{C}_3$  & $\small{\mathrm{Conv}\left\{
\begin{pmatrix}
1 \\
-1 \\
1 \\
\end{pmatrix}, 
\begin{pmatrix}
1 \\
-1 \\
-1 \\
\end{pmatrix}, 
\begin{pmatrix}
1 \\
1 \\
-1 \\
\end{pmatrix}, 
\begin{pmatrix}
1\\
1 \\
1 \\
\end{pmatrix}, 
\begin{pmatrix}
-1\\
1 \\
-1 \\
\end{pmatrix},
\begin{pmatrix}
-1\\
1 \\
1 \\
\end{pmatrix},
\begin{pmatrix}
-1\\
-1 \\
-1 \\
\end{pmatrix},
\begin{pmatrix}
-1\\
-1 \\
1 \\
\end{pmatrix}
\right\}}$ \\
&&\\
 $\mathcal{C}_4$  & $\small{\mathrm{Conv}\left\{
\begin{pmatrix}
2 \\
-1 \\
1 \\
\end{pmatrix}, 
\begin{pmatrix}
2 \\
-1 \\
-1 \\
\end{pmatrix}, 
\begin{pmatrix}
0 \\
1 \\
-1 \\
\end{pmatrix}, 
\begin{pmatrix}
0\\
1 \\
1 \\
\end{pmatrix}, 
\begin{pmatrix}
-1\\
1 \\
-1 \\
\end{pmatrix},
\begin{pmatrix}
-1\\
1 \\
1 \\
\end{pmatrix},
\begin{pmatrix}
-1\\
-1 \\
-1 \\
\end{pmatrix},
\begin{pmatrix}
-1\\
-1 \\
1 \\
\end{pmatrix}
\right\}}$ \\
&&\\
 $\mathcal{C}_5$  & $\small{
\mathrm{Conv}\left\{
\begin{pmatrix}
0 \\
-1 \\
1 \\
\end{pmatrix}, 
\begin{pmatrix}
2 \\
-1 \\
-1 \\
\end{pmatrix}, 
\begin{pmatrix}
2 \\
0 \\
-1 \\
\end{pmatrix}, 
\begin{pmatrix}
0\\
2 \\
1 \\
\end{pmatrix}, 
\begin{pmatrix}
-1\\
0 \\
-1 \\
\end{pmatrix},
\begin{pmatrix}
-1\\
2 \\
1 \\
\end{pmatrix},
\begin{pmatrix}
-1\\
-1 \\
-1 \\
\end{pmatrix},
\begin{pmatrix}
-1\\
-1 \\
1 \\
\end{pmatrix}\right\}}$\\
&&\\
 $\mathcal{D}_1$  & $\small{\mathrm{Conv}\left\{
\begin{pmatrix}
0 \\
1 \\
0 \\
\end{pmatrix}, 
\begin{pmatrix}
0 \\
1 \\
-1 \\
\end{pmatrix}, 
\begin{pmatrix}
3 \\
-1 \\
-1 \\
\end{pmatrix}, 
\begin{pmatrix}
1\\
1 \\
-1 \\
\end{pmatrix}, 
\begin{pmatrix}
-1\\
-1 \\
3 \\
\end{pmatrix},
\begin{pmatrix}
-1\\
-1 \\
-1 \\
\end{pmatrix},
\begin{pmatrix}
-1\\
0 \\
-1 \\
\end{pmatrix},
\begin{pmatrix}
-1\\
0 \\
2 \\
\end{pmatrix}
\right\}}$ \\
&&\\
$\mathcal{D}_2$  & $\small{\mathrm{Conv}\left\{
\begin{pmatrix}
0 \\
1 \\
1 \\
\end{pmatrix}, 
\begin{pmatrix}
0 \\
1 \\
-1 \\
\end{pmatrix}, 
\begin{pmatrix}
2 \\
-1 \\
-1 \\
\end{pmatrix}, 
\begin{pmatrix}
2\\
1 \\
-1 \\
\end{pmatrix},
\begin{pmatrix}
-1\\
-1 \\
2 \\
\end{pmatrix},
\begin{pmatrix}
-1\\
-1 \\
-1 \\
\end{pmatrix},
\begin{pmatrix}
-1\\
0 \\
-1 \\
\end{pmatrix},
\begin{pmatrix}
-1\\
0 \\
2 \\
\end{pmatrix}
\right\}}$\\
\bottomrule
\end{tabular}
\end{center}
\end{table}

\begin{table}[H]
\begin{center}
\begin{tabular}{ccc}
\toprule
Notation  & The moment polytope $\D$  \\ \midrule
\\
$\mathcal{E}_1$  & $\small{\mathrm{Conv}\left\{
\begin{pmatrix}
1 \\
-1 \\
-1 \\
\end{pmatrix}, 
\begin{pmatrix}
1 \\
-1 \\
1 \\
\end{pmatrix}, 
\begin{pmatrix}
0 \\
1 \\
0 \\
\end{pmatrix}, 
\begin{pmatrix}
1\\
0 \\
0 \\
\end{pmatrix},
\begin{pmatrix}
1\\
0 \\
-1 \\
\end{pmatrix}, \right.}$  \\
& &\\
&
 $\small{
 \left. \hspace{5.5cm}
\begin{pmatrix}
0\\
1 \\
-1 \\
\end{pmatrix},
\begin{pmatrix}
-1\\
1 \\
-1 \\
\end{pmatrix},
\begin{pmatrix}
-1\\
1 \\
1 \\
\end{pmatrix},
\begin{pmatrix}
-1\\
-1 \\
3 \\
\end{pmatrix},
\begin{pmatrix}
-1\\
-1 \\
-1 \\
\end{pmatrix}
\right\}}$ \\
& & \\
 $\mathcal{E}_2$  & $\small{\mathrm{Conv}\left\{
\begin{pmatrix}
1 \\
-1 \\
-1 \\
\end{pmatrix}, 
\begin{pmatrix}
1 \\
-1 \\
0 \\
\end{pmatrix}, 
\begin{pmatrix}
0 \\
1 \\
1 \\
\end{pmatrix}, 
\begin{pmatrix}
1\\
0 \\
0 \\
\end{pmatrix},
\begin{pmatrix}
1\\
0 \\
-1 \\
\end{pmatrix}, \right.}$  \\
& &\\
&
 $\small{
 \left. \hspace{5.5cm}
\begin{pmatrix}
0\\
1 \\
-1 \\
\end{pmatrix},
\begin{pmatrix}
-1\\
1 \\
-1 \\
\end{pmatrix},
\begin{pmatrix}
-1\\
1 \\
2 \\
\end{pmatrix},
\begin{pmatrix}
-1\\
-1 \\
2 \\
\end{pmatrix},
\begin{pmatrix}
-1\\
-1 \\
-1 \\
\end{pmatrix}
\right\}}$\\ 
& &\\
 $\mathcal{E}_3$  &
$\small{
\mathrm{Conv}\left\{
\begin{pmatrix}
1 \\
-1 \\
-1 \\
\end{pmatrix}, 
\begin{pmatrix}
1 \\
-1 \\
1 \\
\end{pmatrix}, 
\begin{pmatrix}
0 \\
1 \\
1 \\
\end{pmatrix}, 
\begin{pmatrix}
1\\
0 \\
1 \\
\end{pmatrix},
\begin{pmatrix}
1\\
0 \\
-1 \\
\end{pmatrix}, \right.}$  \\
& &\\
&
 $\small{
 \left. \hspace{5.5cm}
\begin{pmatrix}
0\\
1 \\
-1 \\
\end{pmatrix},
\begin{pmatrix}
-1\\
1 \\
1 \\
\end{pmatrix},
\begin{pmatrix}
-1\\
1 \\
-1 \\
\end{pmatrix},
\begin{pmatrix}
-1\\
-1 \\
1 \\
\end{pmatrix},
\begin{pmatrix}
-1\\
-1 \\
-1 \\
\end{pmatrix}
\right\}}$\\ 
&&\\
 $\mathcal{E}_4$  &
$\small{
\mathrm{Conv}\left\{
\begin{pmatrix}
1 \\
-1 \\
0 \\
\end{pmatrix}, 
\begin{pmatrix}
1 \\
-1 \\
-1 \\
\end{pmatrix}, 
\begin{pmatrix}
0 \\
1 \\
-1 \\
\end{pmatrix}, 
\begin{pmatrix}
1\\
0 \\
-1 \\
\end{pmatrix},
\begin{pmatrix}
1\\
0 \\
1 \\
\end{pmatrix}, \right.}$  \\
& &\\
&
 $\small{
 \left. \hspace{5.5cm}
\begin{pmatrix}
0\\
1 \\
2 \\
\end{pmatrix},
\begin{pmatrix}
-1\\
1 \\
-1 \\
\end{pmatrix},
\begin{pmatrix}
-1\\
1 \\
2 \\
\end{pmatrix},
\begin{pmatrix}
-1\\
-1 \\
-1 \\
\end{pmatrix},
\begin{pmatrix}
-1\\
-1 \\
0 \\
\end{pmatrix}
\right\}}$\\
&&\\
 $\mathcal{F}_1$  &
$\small{
\mathrm{Conv}\left\{
\begin{pmatrix}
0 \\
1 \\
-1 \\
\end{pmatrix}, 
\begin{pmatrix}
0 \\
1 \\
1 \\
\end{pmatrix}, 
\begin{pmatrix}
1 \\
1 \\
-1 \\
\end{pmatrix}, 
\begin{pmatrix}
1\\
0 \\
-1 \\
\end{pmatrix},
\begin{pmatrix}
1\\
0 \\
1 \\
\end{pmatrix},
\begin{pmatrix}
1\\
1\\
1 \\
\end{pmatrix}, \right.}$  \\
& &\\
&
 $\small{
 \left. \hspace{4.5cm}
\begin{pmatrix}
0\\
-1 \\
1 \\
\end{pmatrix},
\begin{pmatrix}
0\\
-1 \\
-1 \\
\end{pmatrix},
\begin{pmatrix}
-1\\
-1 \\
-1 \\
\end{pmatrix},
\begin{pmatrix}
-1\\
-1 \\
1 \\
\end{pmatrix},
\begin{pmatrix}
-1\\
0 \\
1 \\
\end{pmatrix},
\begin{pmatrix}
-1\\
0 \\
-1 \\
\end{pmatrix}
\right\}}$ \\
&&\\
 $\mathcal{F}_2$  &
$\small{
\mathrm{Conv}\left\{
\begin{pmatrix}
0 \\
1 \\
2 \\
\end{pmatrix}, 
\begin{pmatrix}
0 \\
1 \\
-1 \\
\end{pmatrix}, 
\begin{pmatrix}
1 \\
1 \\
2 \\
\end{pmatrix}, 
\begin{pmatrix}
1\\
0 \\
1 \\
\end{pmatrix},
\begin{pmatrix}
1\\
0 \\
-1 \\
\end{pmatrix},
\begin{pmatrix}
1\\
1\\
-1 \\
\end{pmatrix}, \right.}$ \\
& &\\
&
 $\small{
 \left. \hspace{4.5cm}
\begin{pmatrix}
0\\
-1 \\
-1 \\
\end{pmatrix},
\begin{pmatrix}
0\\
-1 \\
0 \\
\end{pmatrix},
\begin{pmatrix}
-1\\
-1 \\
0 \\
\end{pmatrix},
\begin{pmatrix}
-1\\
-1 \\
-1 \\
\end{pmatrix},
\begin{pmatrix}
-1\\
0 \\
-1 \\
\end{pmatrix},
\begin{pmatrix}
-1\\
0 \\
1 \\
\end{pmatrix}
\right\}}\vspace{0.2cm}$ \\ 
\bottomrule
\end{tabular}
\end{center}
\end{table}

\section{Appendix: Corrigenda to our published article \cite{YZ19}}

After publication of our article with Tohoku Math Jounal \cite{YZ19}, some technical errors were unfortunately found when the first author was working in \cite{NSY22}.
In this section, we clarify which part of calculation in Section $\ref{sec:computation}$ is wrong, and explain how this effects the statement of Theorem $\ref{thm:main}$.
Note that this section corresponding to our corrigenda \cite{YZ23}, and we will follow the notation and terminologies used in Sections $\ref{sec:Intro}$-$\ref{sec:example}$.

In Theorem $\ref{prop:unstablecond}$ $(2)$, we established an instability criterion of relative $K$-stability for polarized toric manifolds.
We then applied this sufficient condition to toric Fano threefolds to verify whether given toric Fano threefolds are relatively $K$-stable or not.
Unfortunately, there is a technical mistake in Proposition $5.1$ (b) when we adapted the instability criterion 
\begin{equation}\label{ineq:insta}
1-c < \frac{\int_{\Delta^{\!-}} (1-\theta_\Delta)^2 dx}{\mathrm{Vol}(\Delta^{\!-})}
\tag{1.4}
\end{equation}
to the toric Fano threefold $\mathcal B_1=\C P(\mathcal O_{\C P^2}\oplus \mathcal O_{\C P^2}(2))$, although the instability criterion $\eqref{ineq:insta}$ itself is correct 
This affects the ``only if\," part of Theorem $1.5$ and Corollary $1.6$, and provides inconclusive results for relative $K$-stability of $\mathcal C_2$,
$\mathcal D_1$, $\mathcal D_2$, $\mathcal E_1$ and $\mathcal E_2$. We remark that this erroneous part was already mentioned in the paper \cite[Section $3.3.2$]{NSY22} written by the first author and his collaborators.

To the best of our knowledge, the correct statement about relative $K$-stability of toric Fano threefolds is the following.
\begin{theorem}\cite[Theorem $1.5$]{YZ19}
Let $X$ be a toric Fano threefold. We assume that the Futaki invariant of $X$ does not vanish. Then $X$ is relatively strongly $K$-stable in the toric sense in the anti-canonical class {\textbf{if}}
$X$ is one of the following: $\bf {\mathcal B}_1$, $\mathcal B_2$, $\mathcal B_3$, $\mathcal C_1$, $\mathcal C_4$, $\mathcal E_3$, $\mathcal E_4$ and $\mathcal F_2$.
\end{theorem}
We remark that relative $K$-stability of $(X,-K_X)$ for $X=\mathcal B_1$ is a consequence of the existence of extremal metrics in each K\"ahler class due to Guan \cite{G95}, Hwang \cite{H94},
Apostolov-Calderbank-Gauduchon-T\o nnesen-Friedman \cite{ACGT-F08}. See \cite{NSY22}, Corollary $1.6$ for further details. 
We also mention that the toric Fano threefold $\mathcal C_2$ admits extremal K\"ahler metrics in {\emph{some}} K\"ahler classes due to the work of \cite{BCT-F19}.
However, it is subtle to determine {\emph{which}} K\"ahler class may admit extremal metrics even in the case of its first Chern class.

\begin{table}\label{list-stability}
\begin{center}
Erratum to \cite[Table $1$]{YZ19}. Relative stability in the toric sense of toric Fano threefolds.
\end{center}
\bigskip

\centering
\begin{tabular}{ccc}
\toprule
Notation  & Relative $K$-stability & \\
& &  \\ \midrule
$\C P^3$  & \underline{stable}  
\\ [3pt]
$\mathcal{B}_1$ & {\bf{stable}}  
\\ [3pt] 
$\mathcal{B}_2$  &   stable  
 \\ [3pt]
$\mathcal{B}_3$  &  stable  
\\ [3pt]
$\mathcal{B}_4$ & \underline{stable} 
 \\ [3pt]
$\mathcal{C}_1$ &  stable 
  \\ [3pt]
$\mathcal{C}_2$ & {\bf{inconclusive}} 
\\ [3pt]
$\mathcal{C}_3$ & \underline{stable} 
 \\ [3pt]
$\mathcal{C}_4$  & stable 
  \\ [3pt]
$\mathcal{C}_5$ & \underline{stable} 
\\ [3pt]
$\mathcal{D}_1$ & {\bf{inconclusive}} 
 \\ [3pt]
$\mathcal{D}_2$ & {\bf{inconclusive}}  
\\ [3pt]
$\mathcal{E}_1$ & {\bf{inconclusive}} 
 \\ [3pt]
$\mathcal{E}_2$ & {\bf{inconclusive}} 
\\ [3pt]
$\mathcal{E}_3$ &  stable 
 \\ [3pt]
$\mathcal{E}_4$ &  stable 
\\ [3pt] 
$\mathcal{F}_1$ & \underline{stable} 
\\ [3pt] 
$\mathcal{F}_2$ &  stable  
 \\
\bottomrule
\end{tabular}
\end{table}

\vskip 8pt

\noindent {\bf{The correct values of integration in the criteria $\eqref{ineq:insta}$.}}
Our erroneous parts in the original paper \cite[Table $1$]{YZ19} are caused by a technical mistake when we adapt the computational package \cite{normaliz} to calculate the integration over $\Delta^-$. 
This subsection summarizes the correct values of $\int_{\Delta^{\!-}} (1-\theta_\Delta)^2 dx$ for each toric Fano threefold, ${\mathcal B}_1$, $\mathcal C_2$, $\mathcal D_1$, $\mathcal D_2$, $\mathcal E_1$ and $\mathcal E_2$.

We repeat the description of the subset $\Delta^{\!-}=\set{x\in \Delta|1-\theta_\Delta <0}$ and the potential function $\theta_\Delta=\sum_{i=1}^3a_ix_i+c$ of each toric Fano threefold or the reader's convenience, 
although these data already appeared in Table $\ref{table:toricFano3}$. %\cite[Table $2$]{YZ19}.

Here and hereafter, we only consider relative $K$-stability of toric Fano threefolds with respect to their anticanonical polarization.

\vskip 5pt
%%%%%%%%%%%%%%%%%%%%%%%%B1
\noindent $\bullet$ $\mathcal B_1$. Let $\Delta$ be the moment polytope corresponding to $\mathcal B_1$. Then we see that
\begin{align*}
\theta_\Delta&=-\frac{620}{349}x_3-\frac{240}{349},\\  
\Delta^-&=\mathrm{Conv}\Set{
\begin{pmatrix}
4 \\
-1 \\
-1 \\
\end{pmatrix}, 
\begin{pmatrix}
\frac{39}{10} \\
-1 \\
-\frac{19}{20} \\
\end{pmatrix}, 
\begin{pmatrix}
-1 \\
-1 \\
-\frac{19}{20} \\
\end{pmatrix}, 
\begin{pmatrix}
 -1 \\
\frac{39}{10} \\
-\frac{19}{20} \\
\end{pmatrix},
\begin{pmatrix}
-1 \\
-1 \\
-1 \\
\end{pmatrix},
\begin{pmatrix}
-1 \\
4 \\
-1 \\
\end{pmatrix}},\\
\mathrm{Vol}(\Delta^-)&=\frac{7351}{12000} \qquad \text{and}   \qquad \int_{\Delta^-}(1-\theta_\Delta)^2dx=\frac{23785711}{14616120000}.
\end{align*}
Since the left-hand side value of $\eqref{ineq:insta}$ is $1-c=\frac{589}{349}$, we see that the toric Fano threefold $\mathcal B_1$ is inconclusive.
However, it must be relatively $K$-stable because $\mathcal B_1=\C P(\mathcal O_{\C P^2}\oplus \mathcal O_{\C P^2}(2))$ does admit extremal metrics in all K\"ahler classes as explained in the above. 

\vskip 5pt
%%%%%%%%%%%%%%%%%%%%%%%%%%%C2
\noindent $\bullet$ $\mathcal C_2$. Let $\Delta$ be the moment polytope corresponding to $\mathcal C_2$. Then we see that 
\begin{align*}
\theta_\Delta&=-\frac{7600}{17787}x_1-\frac{17750}{17787}x_3-\frac{4868}{17787},\\
  \Delta^-&= \mathrm{Conv}\Set{
\begin{pmatrix}
-\frac{981}{1520} \\
-1 \\
-1 \\
\end{pmatrix}, 
\begin{pmatrix}
-\frac{981}{1520} \\
\frac{4021}{1520} \\
-1 \\
\end{pmatrix}, 
\begin{pmatrix}
-1 \\
\frac{10111}{3550} \\
-\frac{3011}{3550} \\
\end{pmatrix}, 
\begin{pmatrix}
 -1 \\
3 \\
-1 \\
\end{pmatrix}, 
\begin{pmatrix}
-1 \\
-1 \\
-\frac{3011}{3550} \\
\end{pmatrix},
\begin{pmatrix}
-1 \\
-1 \\
-1 \\
\end{pmatrix}
}, \\
\mathrm{Vol}(\Delta^-)&=\frac{600596677989}{5823363200000} \qquad \text{and}   \qquad \int_{\Delta^-}(1-\theta_\Delta)^2dx=\frac{25255463527}{62892322560000}.
\end{align*}
Since the left-hand side value of $\eqref{ineq:insta}$ is $1-c=\frac{22655}{17787}$, we see that the toric Fano threefold $\mathcal C_2$ is inconclusive.

\vskip 7pt
%%%%%%%%%%%%%%%%%%%%%%%%%%%D1
\noindent $\bullet$ $\mathcal D_1$. Let $\Delta$ be the moment polytope corresponding to $\mathcal D_1$. Then we see that 
\begin{align*}
\theta_\Delta&= \frac{99600}{467581}x_1-\frac{627000}{467581}x_2-\frac{213939}{467581},\\
  \Delta^-&= \mathrm{Conv}\Set{
\begin{pmatrix}
\frac{48388}{18165} \vspace{0.1cm}\\
-\frac{12058}{18165} \\
-1 \\
\end{pmatrix}, 
\begin{pmatrix}
\frac{1363}{2490} \\
-1 \\
\frac{3617}{2490} \\
\end{pmatrix}, 
\begin{pmatrix}
3 \\
-1 \\
-1 \\
\end{pmatrix}, 
\begin{pmatrix}
\frac{1363}{2490} \\
-1 \\
-1 \\
\end{pmatrix}
}, \\
\mathrm{Vol}(\Delta^-)&=\frac{227763307043}{675748899000} \qquad \text{and}   \qquad \int_{\Delta^-}(1-\theta_\Delta)^2dx=\frac{33978139207574189228}{3693508189588076033475}.
\end{align*}
Since the left-hand side value of $\eqref{ineq:insta}$ is $1-c=\frac{681520}{467581}$, we see that the toric Fano threefold $\mathcal D_1$ is inconclusive.

\vskip 7pt
%%%%%%%%%%%%%%%%%%%%%%%%%%%D2
\noindent $\bullet$ $\mathcal D_2$. Let $\Delta$ be the moment polytope corresponding to $\mathcal D_2$. Then we see that 
\begin{align*}
\theta_\Delta&= \frac{219420}{650251}x_1-\frac{318320}{650251}x_2-\frac{62565}{650251},\\
  \Delta^-&= \mathrm{Conv}\Set{
\begin{pmatrix}
2\\
-\frac{1489}{1730} \\
-1 \\
\end{pmatrix}, 
\begin{pmatrix}
\frac{4288}{2385} \\
-1 \\
-\frac{1903}{2385} \\
\end{pmatrix}, 
\begin{pmatrix}
2 \\
-1 \\
-1 \\
\end{pmatrix}, 
\begin{pmatrix}
\frac{4288}{2385} \\
-1 \\
-1 \\
\end{pmatrix}
}, \\
\mathrm{Vol}(\Delta^-)&=\frac{13997521}{14760943875} \qquad \text{and}   \qquad \int_{\Delta^-}(1-\theta_\Delta)^2dx=\frac{13762295011178528}{31206581065640687844375}.
\end{align*}
Since the left-hand side value of $\eqref{ineq:insta}$ is $1-c=\frac{712816}{650251}$, we see that the toric Fano threefold $\mathcal D_2$ is inconclusive.

\vskip 7pt
%%%%%%%%%%%%%%%%%%%%%%%%%%%E1
\noindent $\bullet$ $\mathcal E_1$. Let $\Delta$ be the moment polytope corresponding to $\mathcal E_1$. Then we see that 
\begin{align*}
\theta_\Delta&= -\frac{17020}{19651}x_1-\frac{17020}{19651}x_2-\frac{6845}{19651},\\
  \Delta^-&= \mathrm{Conv}\Set{
\begin{pmatrix}
-\frac{103}{185}\\
-1 \\
-1 \\
\end{pmatrix}, 
\begin{pmatrix}
-\frac{103}{185} \\
-1 \\
-\frac{473}{185} \\
\end{pmatrix}, 
\begin{pmatrix}
-1 \\
-\frac{103}{185} \\
-1 \\
\end{pmatrix}, 
\begin{pmatrix}
 -1 \\
-\frac{103}{185} \\
\frac{473}{185} \\
\end{pmatrix}, 
\begin{pmatrix}
-1 \\
-1 \\
3 \\
\end{pmatrix},
\begin{pmatrix}
-1 \\
-1 \\
-1 \\
\end{pmatrix}
}, \\
\mathrm{Vol}(\Delta^-)&=\frac{6912272}{18994875} \qquad \text{and}   \qquad \int_{\Delta^-}(1-\theta_\Delta)^2dx=\frac{338285458174976}{36675475698849375}.
\end{align*}
Since the left-hand side value of $\eqref{ineq:insta}$ is $1-c=\frac{26496}{19651}$, we see that the toric Fano threefold $\mathcal E_1$ is inconclusive.

\vskip 7pt
%%%%%%%%%%%%%%%%%%%%%%%%%%%E2
\noindent $\bullet$ $\mathcal E_2$. Let $\Delta$ be the moment polytope corresponding to $\mathcal E_2$. Then we see that 
\begin{align*}
\theta_\Delta&= -\frac{2646160}{2735927}x_1-\frac{982960}{2735927}x_2-\frac{692905}{2735927}, \\
  \Delta^-&= \mathrm{Conv}\Set{
\begin{pmatrix}
-\frac{13897}{15035}\\
-1 \\
-1 \\
\end{pmatrix}, 
\begin{pmatrix}
-\frac{13897}{15035} \\
-1 \\
-\frac{28932}{15035} \\
\end{pmatrix}, 
\begin{pmatrix}
-1 \\
-\frac{4447}{5585} \\
-1 \\
\end{pmatrix}, 
\begin{pmatrix}
 -1 \\
-\frac{4447}{5585} \\
2 \\
\end{pmatrix}, 
\begin{pmatrix}
-1 \\
-1 \\
2 \\
\end{pmatrix},
\begin{pmatrix}
-1 \\
-1 \\
-1 \\
\end{pmatrix}}, \\
\mathrm{Vol}(\Delta^-)&=\frac{86882559394}{3787488274875} \qquad \text{and}   \qquad \int_{\Delta^-}(1-\theta_\Delta)^2dx=\frac{2914285427870970948608}{141752364572729422744044375}.
\end{align*}
Since the left-hand side value of $\eqref{ineq:insta}$ is $1-c=\frac{3428832}{2735927}$, we see that the toric Fano threefold $\mathcal E_2$ is inconclusive.


\begin{thebibliography}{9}

\bibitem{ACGT-F08}
V.~Apostolov, D. M. J.~Calderbank, P.~Gauduchon and C. W.~T$\o$ nnesen-Friedman, 
\newblock {\em Hamiltonian $2$-forms in K\"ahler geometry. III. Extremal metrics and stability}, 
\newblock {Invent. Math.} {\bf{173}} (2008), $547$--$601$.



\bibitem{B81}
V. Batyrev,
\newblock {\em Troidal Fano $3$-folds},
\newblock Math. USSR-Izv. {\bf{19}} ($1982$), $13$--$25$.  
\newblock Izv. Akad. Nauk SSSR {\bf{45}} ($1981$), $704$--$717$. 


\bibitem{B98}
V. V. Batyrev,
\newblock {\em On the classification of toric Fano $4$-folds},
\newblock Algebraic geometry, $9$, J. Math. Sci. 
\newblock (New York) {\bf{94}} ($1999$), $1021$--$1050$. 


\bibitem{BS99}
V. Batyrev and E. Selivanova,
\newblock {\em Einstein-K\"ahler metrics on symmetric toric Fano manifolds},
\newblock J. Reine Angew. Math. 
\newblock {\bf{512}} ($1999$), $225$--$236$. 


\bibitem{Ber12}
R. Berman,
\newblock {\em K-stability of Q-Fano varieties admitting Kahler-Einstein metrics},
\newblock  Invent. Math.
\newblock {\bf{203}} ($2016$), $973$--$1025$.


\bibitem{BCT-F19}
C. P.~Boyer, D. M. J.~Calderbank and C. W. T\o nnesen-Friedman, 
\newblock {\em The K\"ahler geometry of Bott manifolds}, 
\newblock {Adv. Math.} {\bf{350}} ($2019$), $1$--$62$.


\bibitem{normaliz}
W. Bruns, B. Ichim, T. R\"omer and C. S\"oger,
\newblock {\em Normaliz}.
\newblock Algorithms for rational cones and affine monoids.
\newblock Available at {\tt{http://www.math.uos.de/normaliz}}. 

\bibitem{Ca82}
E. Calabi, 
\newblock {\em Extremal K\"ahler metrics}, 
\newblock Seminar on differential geometry,  Ann. of Math Studies, Princeton Univ. Press,
{\bf{102}} ($1982$), $259$--$290$.


\bibitem{CLS}
B. H. Chen,   A. M. Li and L. Sheng, 
\newblock {\em Extremal metrics on toric surfaces},
\newblock arXiv:$1008.2607$.

\bibitem {CLW08} 
X. X. Chen, C. LeBrun and B. Weber, 
\newblock {\em On conformally K\"ahler-Einstein manifolds},
\newblock J. Amer. Math. Soc. {\bf{21}} ($2008$), $1137$--$1168$.

\bibitem{CLS11}
D. A. Cox, J. B. Little and H. K. Schenck.
\newblock {\em Toric varieties}, 
\newblock Graduate Studies in Mathematics, {\bf{124}}. 
American Mathematical Society, Providence, RI, $2011$. xxiv+$841$ pp.

\bibitem {Dolga03} 
I. Dolgachev, 
\newblock {\em Lectures on invariant theory}, 
\newblock London Mathematical Society Lecture Note Series,
{\bf{296}}. Cambridge University Press, Cambridge, $2003$.


\bibitem {Deb03} 
O. Debarre, 
\newblock {\em Fano varieties, Higher dimensional varieties and rational points}, 
\newblock (Budapest, $2001$), $93$--$132$,
\newblock Bolyai Society Mathematical Studies {\bf{12}}, Springer, Berlin $2003$.

\bibitem {D01} 
S. K. Donaldson, 
\newblock {\em Scalar curvature and projective embeddings, I}, 
\newblock J. Diff. Geom. {\bf{59}} ($2001$), $479$--$522$.

\bibitem {D02} 
S. K. Donaldson, 
\newblock {\em Scalar curvature and stability of toric varieties}, 
\newblock J. Diff. Geom. {\bf{62}} ($2002$), $289$--$349$.

\bibitem  {D08}
S. K. Donaldson, 
\newblock{\em Extremal metrics on toric surfaces: a continuity
method},
\newblock J. Diff. Geom. {\bf{79}} ($2008$),  $389$--$432$.

\bibitem {D09}
S. K. Donaldson, 
\newblock{\em Constant scalar curvature metrics on toric surfaces},
\newblock Geom. Funct. Anal. {\bf{19}} ($2009$), $83$--$136$. 

\bibitem{Fut04}
A. Futaki, 
\newblock {\em Asymptotic Chow semi-stability and integral invariants},
\newblock Int. J. Math. {\bf{15}} ($2004$), $967$--$979$.

\bibitem{Fut12}
A. Futaki, 
\newblock {\em Asymptotic Chow stability in K\"ahler geometry},
\newblock Fifth International Congress of Chinese Mathematicians. Part $1$, $2$, $139$--$153$, 
AMS/IP Stud. Adv. Math., {\bf{51}}, pt. $1$, $2$, Amer. Math. Soc., Providence, RI, $2012$.

\bibitem{FM95}
A. Futaki and T. Mabuchi,
\newblock {\em Bilinear forms and extremal K\"ahler vector fields associated with K\"ahler class},
\newblock Math. Ann. {\bf{301}} ($1995$), $199$--$210$.

\bibitem{FOS11}
A. Futaki, H. Ono and Y. Sano, 
\newblock {\em Hilbert series and obstructions to asymptotic semistability},
\newblock Adv. Math. {\bf{226}} ($2011$), $254$--$284$.


\bibitem{polymake}
E. Gawrilow and M. Joswig,
\newblock {\em Polymake}
\newblock {Version $2.9.8$-Convex polytopes, polyhedra, simplicial complexes, matroids, fans, and tropical objects},
\newblock Available at {\tt{wwwopt.mathematik.tu-}}
\newblock {\tt{darmstadt.de/polymake/doku.php}}, $1997$-present. 

\bibitem{GKZ94}
I. M. Gelfand, M. M. Kapranov, and A. V. Zelevinsky,
\newblock {\em Discriminants, resultants, and multidimensional determinants},
\newblock Mathematics: Theory \& Applications. Birkh\"auser Boston Inc.,
  Boston, MA, $1994$.


\bibitem{G95}
D.~Guan, 
\newblock {\em Existence of extremal metrics on compact almost homogeneous K\"ahler manifolds with two ends}, 
\newblock {Trans. Amer. Math. Soc.} \textbf{347}  ($1995$), $2255$--$2262$.

\bibitem{H94}
A.~D.~Hwang, 
\newblock {\em On existence of K\"ahler metrics with constant scalar curvature}, 
\newblock {Osaka J. Math.} \textbf{31} ($1994$), $561$--$595$.



\bibitem{KSZ92}
M. M. Kapranov, B. Sturmfels and A. V. Zelevinsky,
\newblock {\em Chow polytopes and general resultants},
\newblock Duke Math. J. {\bf{67}} ($1992$), $189$--$218$.

\bibitem{Kasp10} 
A. M. Kasprzyk,  
\newblock {\em Canonical toric Fano threefolds}, 
\newblock  Canad. J. Math. {\bf{62}} ($2010$), $1293$--$1309$.

\bibitem{Mab04}
T. Mabuchi,
\newblock {\em Stability of extremal K\"ahler manifolds}, 
\newblock Osaka J. Math. {\bf{41}} ($2004$), $563$--$582$. 

\bibitem{Mab05}
T. Mabuchi,
\newblock {\em An energy-theoretic approach to the Hitchin-Kobayashi correspondence for manifolds. I}, 
\newblock Invent. Math. {\bf{159}} ($2005$), $225$--$243$.

\bibitem{Mab14}
T. Mabuchi,
\newblock {\em Relative stability and extremal metrics},  
\newblock J. Math. Soc. Japan {\bf{66}} ($2014$), $535$--$563$. 

\bibitem{Mab16}
T. Mabuchi,
\newblock {\em Asymptotic polybalanced kernels on extremal Kaehler manifolds},
\newblock  arXiv:$1610.09632$.

\bibitem{Naka98}
Y. Nakagawa,
\newblock {\em Combinatorial formulae for Futaki characters and generalized killing forms
of toric Fano orbifolds},
\newblock The Third Pacific Rim Geometry Conference (Seoul, $1996$), $223$--$260$, 
\newblock Monogr. Geom. Topology, $25$, Int. Press, Cambridge, MA, $1998$.

\bibitem{NiThesis}
B. Nill,
\newblock {\em Gorenstein toric Fano varieties},
\newblock Dissertation, Universit\"at T\"ubingen, $2005$.
\newblock Available at {\tt{http://tobias-lib.uni-tuebingen.de/volltexte/2005/1888}}

\bibitem{NP11}
B. Nill and A. Paffenholz,
\newblock{\em Examples of K\"ahler-Eisntein toric Fano manifolds associated to non-symmetric reflexive polytopes}, 
\newblock Beitr. Algebra. Geom {\bf{52}} ($2011$), $297$--$304$.

\bibitem{NSY22}
Y.~Nitta, S.~Saito and N.~Yotsutani,
\newblock {Relative Ding and $K$-stability of toric Fano manifolds in low dimensions,}
\emph{Eur. J. Math.} ($2023$), 9:29.
{\tt{https://rdcu.be/datku}}


\bibitem{Ob07} 
M. Obro,  
\newblock {\em An algorithm for the classification of smooth Fano polytopes}, 
\newblock  arXiv:$0704.0049$ ($2007$).

\bibitem{Oda12}
Y. Odaka,
\newblock {\em The Calabi conjecture and K-stability},
\newblock Int. Math. Res. Not. IMRN $2012$, $2272$--$2288$.

\bibitem{Ono11}
H. Ono,
\newblock{\em A necessary condition for Chow semistability of polarized toric manifolds}, 
\newblock J. Math. Soc. Japan. {\bf{63}} ($2011$), $1377$--$1389$.

\bibitem{Ono13}
H. Ono,
\newblock{\em Algebro-geometric semistability of polarized toric manifolds}, 
\newblock Asian J. Math. {\bf{17}} ($2013$), $609$--$616$. 


\bibitem{OSY12}
H. Ono, Y. Sano and N. Yotsutani,
\newblock{\em An example of an asymptotically Chow unstable manifold with constant scalar curvature},
\newblock Ann. Inst. Fourier (Grenoble) {\bf{62}} ($2012$), $1265$--$1287$.


\bibitem{RT07}
J. Ross and R. Thomas, 
\newblock{\em A study of the Hilbert-Mumford criterion for the stability of projective varieties}, 
\newblock J. Alg. Geom. 
\newblock 
{\bf{16}} ($2007$), $201$--$255$.


\bibitem{S16}
R. Seyyedali,
\newblock{\em Relative Chow stability and extremal metrics},
\newblock  Adv. Math. {\bf{316}} ($2017$), $770$--$805$.

\bibitem{SZ12}
Y. L. Shi and H. X. Zhu,
\newblock {\em K\"ahler-Ricci solitons on toric Fano orbifolds},
\newblock Math. Zeit. {\bf{271}} ($2012$), $1241$--$1251$.

\bibitem{SzThesis} 
G. Sz\'{e}kelyhidi,  
\newblock {\em Extremal metrics and K-stability}, 
\newblock  Dissertation, Imperial college, London, $2006$. arXiv:$0611002$.

\bibitem{Sz07} 
G. Sz\'{e}kelyhidi,  
\newblock {\em Extremal metrics and K-stability}, 
\newblock  Bull. London Math. Soc. {\bf{39}} ($2007$), $76$--$84$.

\bibitem{Ti97} 
G. Tian, 
\newblock {\em K\"ahler-Einstein metrics with positive scalar curvature},
\newblock Invent. Math. {\bf{130}} ($1997$), $1$--$39$.


\bibitem{TZ02}
 G. Tian and X. H. Zhu, 
\newblock {\em A new holomorphic invariant and uniqueness of Kahler-Ricci solitons},
\newblock Comm. Math. Helv. {\bf{77}} ($2002$), $297$--$325$.

\bibitem{VZ12}
 A. D. Vedova and F. Zuddas, 
\newblock {\em  Scalar curvature and asymptotic Chow stability of projective bundles and blowups},
\newblock Trans. Amer. Math. Soc. {\bf{364}} ($2012$), $6495$--$6511$.

\bibitem{WW82}
K. Watanabe and M. Watanabe,
\newblock {\em The classification of Fano $3$-folds with torus embeddings},
\newblock Tokyo J. Math. {\bf{5}} ($1982$), $37$--$48$.

\bibitem{WZ11}
X. J. Wang and B. Zhou, 
\newblock {\em Existence and nonexistence of extremal metrics on toric K\"ahler manifolds}, 
Adv. Math. {\bf{226}} ($2011$), $4429$--$4455$.


\bibitem{WZ04}
X. J. Wang and H. X. Zhu, 
\newblock {\em K\"ahler-Ricci solitons on toric manifolds with positive first Chern class}, 
Adv. Math. {\bf{188}} ($2004$), $87$--$103$. 

\bibitem{Yotsu15}
N. Yotsutani,
\newblock {\em Facets of secondary polytopes and Chow stability of toric varieties},
\newblock Osaka J. Math {\bf{53}}, ($2016$), $751$--$765$. 

\bibitem{YZ19}
N.~Yotsutani and B.~Zhou,
\newblock {Relative algebro-geometric stabilities of toric manifolds,}
\newblock {\it Tohoku Math. J.} {\bf{71}} $(2019)$, $495$--$524$.

\bibitem{YZ23}
N.~Yotsutani and B.~Zhou,
\newblock {\emph{Corrigenda: Relative Algebro-Geometric stabilities of toric manifolds (Tohoku Math J. {\textbf{71}}, ($2019$), $495$--$524$)}.}
\newblock Tohoku Math. J. {\textbf{75}}, ($2023$), $1$--$5$.


\bibitem{ZZ08}
B. Zhou and X. H. Zhu,
\newblock {\em Relative $K$-stability and modified $K$-energy on toric manifolds},
\newblock Adv. Math. {\bf{219}} ($2008$), $1327$--$1362$.

\bibitem{ZZ08-2}
B. Zhou and X. H. Zhu,
\newblock {\em $K$-stability on toric manifolds},
\newblock  Proc. Amer. Math. Soc. {\bf{136}} ($2008$), $3301$--$3307$.


\bibitem{graded_ring}
\newblock {Graded Ring Database},
\newblock {\tt http://grdb.lboro.ac.uk/forms/toricsmooth}.

\end{thebibliography}
\end{document}